%% file: main_paper.tex
\documentclass{article}

\usepackage{microtype}
\usepackage{graphicx}
\usepackage{booktabs} 
\usepackage[utf8]{inputenc}
\usepackage{xcolor}
\usepackage{amsmath}

\usepackage{xparse}
\usepackage{todonotes}

\usepackage{enumerate}
\usepackage{enumitem}
\usepackage[nottoc,notlof,notlot]{tocbibind} 
\renewcommand{\thesection}{\arabic{section}}
\makeatletter
\let\l@section\l@chapter
\makeatother
\usepackage{amssymb}
\usepackage{amsthm}
\usepackage{url}
\usepackage{graphicx}
\usepackage{tcolorbox}
\usepackage{float}
\usepackage{subcaption}

\restylefloat{figure}
\usepackage{upgreek}
\graphicspath{{C:/Users/Besitzer/Desktop}}
\usepackage{thm-restate}
\newtheorem{theorem}{Theorem}

\newtheorem{corollary}{Corollary}[theorem]
\usepackage{thm-restate}
\newtheorem{lemma}[theorem]{Lemma}
\newtheorem{Proposition}[theorem]{Proposition}
\newtheorem{Assumption}[theorem]{Assumption}
\newcommand{\bigO}{\mathcal{O}}
\newcommand{\hide}[1]{{}}

\newcommand{\idlow}[1]{\mathord{\mathcode`\-="702D\it #1\mathcode`\-="2200}}
\usepackage{hyperref}

\renewcommand{\paragraph}[1]{\vspace{-0.1em} \noindent \textbf{#1}}

\usepackage[accepted]{icml2021}

\icmltitlerunning{Communication-Efficient Distributed Optimization with Quantized Preconditioners}

\begin{document}

\twocolumn[
\icmltitle{Communication-Efficient Distributed Optimization \\ with Quantized Preconditioners}



\icmlsetsymbol{equal}{*}

\begin{icmlauthorlist}
\icmlauthor{Foivos Alimisis}{unige}
\icmlauthor{Peter Davies}{ist}
\icmlauthor{Dan Alistarh}{ist,neurmagic}
\end{icmlauthorlist}

\icmlaffiliation{unige}{Department of Mathematics, University of Geneva, Switzerland (work done while at IST Austria)}
\icmlaffiliation{ist}{IST Austria}
\icmlaffiliation{neurmagic}{Neural Magic, US}

\icmlcorrespondingauthor{Foivos Alimisis}{Foivos.Alimisis@unige.ch}

\icmlkeywords{Distributed Optimization, Bit Complexity, Second Order Algorithms}

\vskip 0.3in
]



\printAffiliationsAndNotice{}  

\begin{abstract}
We investigate fast and communication-efficient algorithms for the classic problem of minimizing a sum of strongly convex and smooth functions  that are distributed among $n$ different nodes, which can communicate using a limited number of bits. Most previous communication-efficient approaches for this problem are limited to first-order optimization, and therefore have \emph{linear} dependence on the condition number in their communication complexity. 
We show that this dependence is not inherent: communication-efficient methods can in fact have sublinear dependence on the condition number. 
For this, we design and analyze the first communication-efficient distributed variants of preconditioned gradient descent for Generalized Linear Models, and for Newton's method. 
Our results rely on a new technique for quantizing both the preconditioner and the descent direction at each step of the algorithms, while controlling their convergence rate. 
We also validate our findings experimentally, showing fast convergence and reduced communication. 
\end{abstract}
\vspace{-4mm}
\input{01_introduction.tex}

\vspace{-2mm}
\input{04_quantized_gradient_GLM.tex}

\vspace{-2mm}
\input{05_quantized_newton.tex}

\vspace{-2mm}
\input{06_function_estimation.tex}
\vspace{-2mm}
\input{07_experiments.tex}
\vspace{-2mm}
\input{08_discussion.tex}

\bibliography{refs}
\bibliographystyle{icml2021}

\input{appendix.tex}

\end{document}

%% file: 01_introduction.tex
\section{Introduction}

Due to the sheer size of modern datasets, many practical instances of large-scale optimization are now \emph{distributed}, in the sense that data and computation are split among several computing nodes, which collaborate to jointly optimize the global objective function. 
This shift towards distribution induces new challenges, and many classic algorithms have been revisited to reduce  distribution costs. These costs are usually measured in terms of the number of bits sent and received by the nodes (\emph{communication complexity}) or by the number of parallel iterations required for convergence (\emph{round complexity}).

In this paper, we focus on the communication (bit) complexity of the classic empirical risk minimization problem
\vspace{-1mm}
\begin{equation*}
    \min_{x\in \mathbb{R}^d} f(x) := \frac{1}{n} \sum_{i=1}^n f_i(x),
    \vspace{-1mm}
\end{equation*}
where the global $d$-dimensional cost function $f$ is formed as the average of smooth and strongly-convex local costs $f_i$, each owned by a different machine, indexed by $i=1,...,n$. 

This problem has a rich history. 
The seminal paper of \citet{4048825} considered the case $n = 2$, and provided a lower bound of $\Omega( d \log (d / \epsilon ))$ for quadratic functions, as well as an almost-matching upper bound for this case, within logarithmic factors. (Here, $d$ is the problem dimension and $\epsilon$ is the  error-tolerance.) 

The problem has concentrated significant attention, given the surge of interest in distributed optimization and machine learning, e.g.~\cite{niu2011hogwild, jaggi2014communication, alistarh2016qsgd, nguyen2018sgd, ben2019demystifying}. In particular, a  series of  papers~\cite{khirirat2018distributed, ye2018communication, magnusson2020maintaining, alistarh2020improved} continued to provide improved upper and lower bounds for the communication complexity of this problem, both for deterministic and randomized algorithms, as well as examining related distributed settings and problems~\cite{scaman2017optimal, jordan2018communication, vempala2020communication, mendler2020randomized, hendrikx2020statistically}. 

The best known lower bound for solving the above problem for deterministic algorithms and general $d$ and $n$ is of 
\vspace{-1mm}
\[
    \Omega( n d \log ( d / \epsilon ) ) 
    \vspace{-1mm}
\]
total communication bits, given recently by~\cite{alistarh2020improved}. 
This lower bound can be asymptotically matched for quadratic functions by a quantized variant of gradient descent~\cite{magnusson2020maintaining, alistarh2020improved} using 
\vspace{-1mm}
\[
    \bigO( n d \kappa \log \kappa \log (\gamma d / \epsilon))
    \vspace{-1mm}
\]
total bits, where $\kappa$ is the condition number of the problem and $\gamma$ is the smoothness bound of $f$. 

An intriguing open question concerns the optimal dependency on the condition number for general objectives. While existing lower bounds show no such explicit dependency, all known algorithms have linear (or worse) dependency on $\kappa$. Resolving this problem is non-trivial, since one usually removes this dependency in the non-distributed case by leveraging curvature information in the form of preconditioning or full Newton steps. 
However, existing distribution techniques are designed for \emph{gradient} quantization, and it is not at all clear for instance how using a preconditioning matrix would interact with the convergence properties of the algorithm, and in particular whether favourable convergence behaviour can be preserved at all following quantization. 

\paragraph{Contribution.} In this paper, we resolve this question in the positive, and present  communication-efficient variants of preconditioned gradient descent for generalized linear models (GLMs) and distributed Newton's method.  

Specifically, given a \emph{small enough}  error-tolerance $\epsilon$, a communication-efficient variant of preconditioned gradient descent for GLMs (QPGD-GLM) can find an $\epsilon$-minimizer of a $\gamma$-smooth function using a total number of bits 
\begin{equation*}
        B^{\idlow{QPGD-GLM}} = \bigO \left(n d \kappa_{\ell} \log (n \kappa_{\ell} \kappa(M)) \log (\gamma D/\epsilon)\right), 
\end{equation*}
where $d$ is the dimension, $n$ is the number of nodes,  $\kappa_\ell$ is the condition number of the loss function $\ell$ used to measure the distance of training data from the prediction, $\kappa(M)$ is the condition number of the averaged covariance matrix of the training data, and $D$ is a bound on the initial distance from the optimum. In practice, $\kappa_{\ell}$ is often much smaller than the condition number $\kappa$ of the problem, and is equal to $1$ in the case that $\ell$ is a quadratic. 

This first result suggests that distributed methods need not have linear dependence on the condition number of the problem. 
Our main technical result extends the approach to a distributed variant of Newton's method, showing that the same problem can be solved using 
\begin{equation*}
   B^{\idlow{Newton}} = \bigO\left(nd^2 \log \left( {d} \kappa \right) \log  (\gamma \mu/\sigma \epsilon) \right) \textnormal{ total bits, }
\end{equation*}
under the assumption that the Hessian is $\sigma$-Lipschitz. 

Viewed in conjunction with the above $\Omega(nd \log(d / \varepsilon))$ lower bound, these algorithms outline a new communication complexity trade-off between the dependency on the dimension of the problem $d$, and its  condition number $\kappa$. Specifically, for ill-conditioned but low-dimensional problems, it may be advantageous to employ quantized Newton's method, whereas QPGD-GLM can be used in cases where the structure of the training data favors preconditioning. Further, our results suggest that there can be no general  communication lower bound with linear dependence on the condition number of the problem.

Our results assume the classic coordinator / parameter server~\citep{PS} model of distributed computing, in which a distinguished node acts as a coordinator by gathering model updates from the nodes. 
In this context, we introduce a few tools which should have broader applicability. 
One is a lattice-based matrix quantization technique, which extends the state-of-the-art vector (gradient) quantization techniques to  preconditioners. 
This enables us to carefully trade off the  communication compression achieved by the algorithm with the non-trivial error in the descent directions due to quantization. 
Our main technical advance is in the context of quantized Newton's method, where we need to keep track of the concentration of quantized Hessians relative to the full-precision version. Further, our algorithms quantize directly the local descent directions obtained by multiplying the inverse of the quantized estimation of the preconditioner with the exact local gradient. This is a non-obvious choice, which turns out to be the correct way to deal with quantized preconditioned methods.

We validate our theoretical results on standard regression datasets, where we show that our techniques can provide an improvement of over $3 \times$ in terms of total communication complexity used by the algorithm, while maintaining convergence and solution quality.

\paragraph{Related Work.} 
There has been a surge of interest in distributed optimization and machine learning. While a complete survey is beyond our scope, we mention the significant work on designing and analyzing communication-efficient versions of classic optimization algorithms, e.g.~\cite{jaggi2014communication, scaman2017optimal, jordan2018communication, khirirat2018distributed, nguyen2018sgd, alistarh2016qsgd, alistarh2018convergence, ye2018communication, NUQSGD,  magnusson2020maintaining, ghadikolaei2020communication}, and the growing interest in communication and round complexity lower bounds, e.g.~\cite{NIPS2013_4902, NIPS2014_5386, NIPS2015_5731, vempala2020communication, alistarh2020improved}. In this context, our work is among the first to address the bit complexity of optimization methods which explicitly employ curvature information, and shows that such methods can indeed be made communication-efficient.

\citet{4048825} gave the first upper and lower bounds for the communication (bit) complexity of distributed convex optimization, considering the case of two nodes. Their algorithm is a variant of gradient descent which performs \emph{adaptive} quantization, in the sense that nodes adapt the number of bits they send and the quantization grid depending on the iteration. 
Follow-up work, e.g.~\cite{khirirat2018distributed, alistarh2016qsgd} generalized their algorithm to an arbitrary number of nodes, and continued to improve complexity.

In this line, the work closest to ours is that of \citet{magnusson2020maintaining}, who introduce a family of adaptive gradient quantization schemes which can enable linear convergence in any norm for gradient-descent-type algorithms, in the same system setting considered in our work. 
However, we emphasize that this work did \emph{not} consider preconditioning.
(\citet{alistarh2020improved} also focus on GD, but use different quantizers and a more refined analysis to obtain truly tight communication bounds for quadratics.)


%
Conceptually, the quantization techniques we introduce serve a similar purpose---to allow the convergence properties of the algorithm to be preserved, despite noisy directional information. 
At the technical level, however, the schemes we describe and analyze are different, and arguably more complex. For instance, since only the gradient information is quantized,~\citep{magnusson2020maintaining} can use grid quantization adapted to gradient norms, whereas employ more complex quantization, as well as fine-grained bookkeeping with respect to the concentration of quantized matrices and descent directions. 

There has been significant work on distributed approximate second-order methods with the different goal of minimizing the \emph{number of communication rounds} required for convergence. 
One of the first such works is~\cite{shamir2014communication}, who considered the strongly convex case, and proposed a method called DANE, where each worker solves a subproblem using full gradients at each iteration, and the global iterate is the average of these sub-solutions. 
Follow-up work~\cite{zhang2015disco, reddi2016aide, wang2018giant, zhang2020distributed} proposed improvements  both in terms of generalizing the structure of the loss functions, but also in terms of convergence rates. 
Recently,~\citet{hendrikx2020statistically} also proposed a round-efficient distributed preconditioned accelerated gradient method for our setting, where preconditioning is done by solving a local optimization problem over a subsampled dataset at
the server. Their convergence rate depends on the square root of the relative condition number between the global and local loss functions. 

Concurrent work by~\citet{islamov2021distributed} considers the same problem of reducing the bit cost of distributed second-order optimization, and proposes a series of algorithms based on the novel idea of \emph{learning parameters of the Hessian at the optimum} in a communication-efficient manner. 
The resulting algorithms allow for $\ell_2$-regularization, and can achieve local linear and superlinear rates, independent of the condition number, with \emph{linear} communication cost per round in the dimension $d$. 

Relative to our setting, their results require two additional assumptions: The first is that, for linear communication cost, either the coordinator must have access to \emph{all the training data} at the beginning of the optimization process, or the data should be highly \emph{sparse}.   
The second assumption is  on the structure of the individual loss functions, which are weaker than the assumptions we make for our ``warm-up'' algorithm for GLMs, but stronger than the ones required for our generalized quantized Newton's method.  
Their results are therefore not directly comparable to ours, however, we note that our communication cost should be lower in e.g. the case where the data is dense and the number of points $m$ is larger than the dimension $d$. The algorithmic techniques are rather different. 
Follow-up work extended their approach to the federated learning setting~\citep{safaryan2021fednl}.

%% file: 04_quantized_gradient_GLM.tex
\section{Preliminaries}
\label{sec:background}

\paragraph{Distributed Setting.}
As discussed, we are in a standard distributed optimization setting, where we have $n$ nodes, and each node $i$ has its own local cost function $f_i: \mathbb{R}^d \rightarrow \mathbb{R}$ (where $d$ is the dimension of the problem). We wish to minimize the average cost 
$f=\frac{1}{n} \sum_{i=1}^n f_i$
and, for that, some communication between nodes is required. We denote the (unique) minimizer of $f$ by $x^*$ and the (unique) minimizer of each $f_i$ by $x_i^*$ (minimizers are unique since these functions are assumed to be strongly convex). Communication may be performed over various network topologies, but in this work we assume a simple structure where an arbitrary node plays the role of the central server, i.e. receives messages from the others, processes them, and finally sends the result back to all. (Such topologies are also common in practice~\cite{PS}.)
Then, the nodes compute an update based on their local cost, and subsequently transmit this information again to the master, repeating the pattern until convergence. 

The two main usually considered complexity metrics are the total number of rounds, or iterations, which the algorithm requires, and the total number of bits transmitted. 
In this paper, we focus on the latter metric, and assume that nodes cannot communicate their information with infinite precision, but instead aim to limit the number of bits that each node can use to encode messages. Thus, we measure complexity in terms of the total number of bits that the optimization algorithm needs to use, in order to minimize $f$ within some accuracy.

\paragraph{Matrix Vectorization.} One of the main technical tools of our work is quantization of matrices. All the matrices that we care to quantize turn out to be symmetric. The first step for quantizing is to vectorize them. We do so by using the mapping
\vspace{-1mm}
\begin{equation*}
    \phi: \mathbb{S}(d) \rightarrow \mathbb{R}^{\frac{d(d+1)}{2}}
\end{equation*}
\vspace{-1mm}
defined by
\begin{equation*}
    \phi(P)=(p_{11},...,p_{1d},p_{22},...,p_{2d},...,p_{dd}),
    \vspace{-1mm}
\end{equation*}
where $P=(p_{ij})_{i,j=1}^d$ and $\mathbb{S}(d)$ is the space of $d \times d$ symmetric matrices. Thus, the mapping $\phi$ just isolates the upper triangle of a symmetric matrix and writes it as a vector. It is direct to check that $\phi$ is a linear isomorphism ($\textnormal{dim}(\mathbb{S}(d))=d(d+1)/2$).
\newline
We can now bound the deformation of distances produced by this mapping for the $\ell_2$ norm in $\mathbb{S}(d)$ and the $\ell_2$ one in $\mathbb{R}^{\frac{d(d+1)}{2}}$:
\begin{restatable}{lemma}{basic}
\label{le:norm_distortion}
For any matrices $P,P' \in \mathbb{S}(d)$, we have
\begin{equation*}
    \frac{1}{\sqrt{d}}\| \phi(P)-\phi(P') \|_2 \leq \|P-P'\|_2 \leq \sqrt{2} \| \phi(P)-\phi(P') \|_2.
\end{equation*}
\label{prop:basic}
\end{restatable}
\vspace{-4mm}
The proof can be found in Appendix \ref{app:isomorphism}.
\newline
We will use the isomorphism $\phi$ later in our applications to Generalized Linear Models and Newton's method. This is the reason of appearance of the extra $d$ \textit{inside a logarithm} in our upper bounds. From now on we use $\| \cdot \|$ to denote the $\ell_2$ norm of either vectors or matrices.

\paragraph{Lattice Quantization.}
For estimating the gradient and Hessian in a distributed manner with limited communication, we use a quantization procedure developed in \cite{davies2021new}. The original quantization scheme involves randomness, but we use a \textit{deterministic} version of it, by picking up the closest point to the vector that we want to encode. This is similar to the quantization scheme used by \cite{alistarh2020improved} for standard gradient descent, and has the following properties:
\begin{Proposition} \cite{davies2021new,alistarh2020improved}
\label{lattice_quantization}
Denoting by $b$ the number of bits that each machine uses to communicate, there exists a quantization function
\vspace{-1mm}
\begin{equation*}
    Q: \mathbb{R}^d \times \mathbb{R}^d\times 	\mathbb{R_+} \times \mathbb{R_+} \rightarrow \mathbb{R}^d,
    \vspace{-1mm}
\end{equation*}
which, for each $\epsilon,y>0$, 
consists of an encoding function
$\textnormal{enc}_{\epsilon,y}:\mathbb{R}^d \rightarrow \lbrace 0,1 \rbrace ^b$ and a decoding one $\textnormal{dec}_{\epsilon,y}:\lbrace 0,1 \rbrace ^b \times \mathbb{R}^d \rightarrow \mathbb{R}^d$, such that, for all $x, x' \in \mathbb{R}^d$,
\vspace{-2mm}
\begin{itemize}
    \item $\textnormal{dec}_{\epsilon,y} (\textnormal{enc}_{\epsilon,y}(x),x') = Q(x,x', y ,\epsilon)$, if $\|x-x'\| \leq y$.
    \vspace{-6mm}
    \item $\|Q(x,x',y,\epsilon)-x\| \leq \epsilon$, if $\|x-x'\| \leq y$.
    \vspace{-2mm}
    \item If $y/\epsilon>1$, the cost of the quantization procedure in number of bits satisfies $b= \bigO(d \textnormal{log}_2 \left(\frac{y}{\epsilon})\right)$.
\end{itemize}
\end{Proposition}

\section{Quantized Preconditioned Gradient Descent for GLMs}
\label{sec:GLMs}

As a warm-up, we consider the case of a Generalized Linear Model (GLM) with data matrix $A \in \mathbb{R}^{m \times d}$. GLMs are particularly attractive models to distribute, because the distribution across nodes can
be performed naturally by partitioning the available data. For more background on distributing GLMs  see~\citep{mendler2020randomized}. 

The matrix $A$ consists of the data used for training in its rows, i.e. we have $m$-many $d$-dimensional data points. As is custom in regression analysis, we assume that $m \gg d$, i.e. we are in the case of big but low-dimensional data. If $m$ is very large, it can be very difficult to store the whole matrix $A$ in one node, so we distribute it in $n$-many nodes, each one owning $m_i$-many data points ($m=\sum_{i=1}^n m_i)$. We pack the data owned by node $i$ in a matrix $A_i \in \mathbb{R}^{m_i \times d}$ and denote the function used to measure the error on machine $i$ by $\ell_i: \mathbb{R}^{m_i} \rightarrow \mathbb{R}$. Then the local cost function $f_i:\mathbb{R}^d \rightarrow \mathbb{R}$ at machine $i$ reads
\vspace{-1mm}
\begin{equation*}
    f_i(x)=\ell_i(A_ix).
    \vspace{-1mm}
\end{equation*}
We can express the global cost function $f$ in the form
\vspace{-1mm}
\begin{equation*}
    f(x)=\ell(Ax)
    \vspace{-1mm}
\end{equation*}
where $\ell:\mathbb{R}^m \rightarrow \mathbb{R}$ is a global loss function defined by
\vspace{-1mm}
\begin{equation*}
    \ell(y)=\frac{1}{n} \sum_{i=1}^n \ell_i(y_i),
    \vspace{-1mm}
\end{equation*}
where $y_i$ are sets of $m_i$-many coordinates of $y$ obtained by the same data partitioning.
\begin{Assumption}
The local loss functions $\ell_i$ are $\mu_{\ell}$-strongly convex and $\gamma_{\ell}$-smooth.
\end{Assumption}
This assumption implies that the global loss function $\ell$ is $\frac{\mu_{\ell}}{n} $-strongly convex and $\frac{\gamma_{\ell}}{n}$-smooth. This is because the Hessian of $\ell$ has the block-diagonal structure
\begin{equation*}
    \nabla_y^2 \ell(y)=\frac{1}{n} \textnormal{diag} \left(\nabla_{y_1}^2 \ell_1(y_1),..., \nabla_{y_n}^2 \ell_n(y_n) \right)
\end{equation*}
and the eigenvalues of all matrices $\nabla_{y_i}^2 \ell_i(y_i)$ are between $\mu_{\ell}$ and $\gamma_{\ell}$. The Hessian of $f$ can be written as
\begin{equation*}
    \nabla^2 f(x)=A^T \nabla^2 \ell(Ax) A \in \mathbb{S}(d) \subseteq \mathbb{R}^{d \times d}.
\end{equation*}
We detail the computation of $\nabla^2 f$ in Appendix \ref{app:GLM_technicalities}.

\begin{Assumption}
The matrix $A \in \mathbb{R}^{m \times d}$ is of full rank (i.e. $rank(A)=d$, since $d<m$).
\end{Assumption}
This assumption is natural: if two columns of the matrix $A$ were linearly dependent, we would not need both the related features in our statistical model. Practically, we can prune one of them and get a new data matrix of full-rank. 
\begin{restatable}{Proposition} {improvedconditionnumber}
\label{prop:improved_condition_number}
The maximum eigenvalue $\lambda_{max}$ of $\nabla^2 f$ satisfies
\vspace{-1mm}
\begin{equation*}
    \gamma:=\lambda_{max}(\nabla^2 f) \leq \gamma_{\ell} \lambda_{max}\left(\frac{A^T A}{n} \right)
\end{equation*}
\vspace{-1mm}
and the minimum eigenvalue $\lambda_{min}$ of $\nabla^2 f$ satisfies
\vspace{-1mm}
\begin{equation*}
    \mu:=\lambda_{min}(\nabla^2 f) \geq \mu_{\ell} \lambda_{min}\left(\frac{A^T A}{n} \right) .
\end{equation*}
\end{restatable}
\vspace{-1mm}
The proof is presented in Appendix \ref{app:GLM_technicalities}.
Thus, we have that the condition number $\kappa$ of our minimization problem satisfies
\vspace{-1mm}
\begin{equation*}
    \kappa \leq \kappa_{\ell} \kappa\left(\frac{A^T A}{n} \right),
    \vspace{-1mm}
\end{equation*}
where $\kappa\left(\frac{A^T A}{n} \right)$ is the condition number of the covariance matrix $A^T A$ averaged in the number of machines. 
The convergence rate of gradient descent generally depends on $\kappa$, which can be much larger than $\kappa_{\ell}$ in case that the condition number of $A^T A$ is large. The usual way to get rid of $\kappa\left(\frac{A^T A}{n} \right)$ is to precondition gradient descent using $\frac{A^T A}{n}$, which we denote by $M$ from now on (we recall the convergence analysis of this method in Appendix \ref{app:precond_gradient}). In our setting $M$ is not known to all machines simultaneously, since each machine owns only a part of the overall data. However, we observe that
\vspace{-1mm}
\begin{equation*}
    M= \frac{1}{n} \sum_{i=1}^n A_i^T A_i,
    \vspace{-1mm}
\end{equation*}
where $A_i^T A_i =: M_i$ is the local covariance matrix of the data owned by the node $i$. 

\subsection{The Algorithm}
In this section we present our QPGD-GLM algorithm and study its communication complexity. 
We structure the algorithm in four steps: first, we describe how to recover a quantized version of the averaged covariance matrices. 
Then, we describe how nodes perform initialization. Next, we describe how nodes can quantize the initial descent direction.
Finally, we describe how to quantize the descent directions for subsequent steps. Our notation for quantization operations follows Section~\ref{sec:background}. 

\noindent\rule[0.5ex]{\linewidth}{1pt}
\vspace{-8mm}
\begin{enumerate}
\item Choose an arbitrary master node, say $i_0$.
\subsection*{(A) Averaged Covariance Matrix Quantization:}
\item Compute $M_i:=A_i^T A_i$ in each node.
\vspace{-2mm}
\item Encode $M_i$ in each node $i$ and decode it in the master node using its information:
\vspace{-2mm}
\begin{align*}
\resizebox{0.9\hsize}{!}{$\bar M_i = \phi^{-1}\left(Q\left(\phi(M_i), \phi(M_{i_0}) , 2 \sqrt{d} n \lambda_{max}(M) , \frac{\lambda_{min}(M)}{16 \sqrt{2} \kappa_{\ell}} \right) \right)$}.
\end{align*}
\vspace{-5mm}

In detail, we first transform the local matrix $M_i$ via the isomorphism $\phi$,  and then quantize it via $Q$, with carefully-set parameters. The matrix will be then de-quantized relative to the master's reference point $\phi(M_{i_0})$, and then re-constituted (in approximate form) via the inverse isomorphism. 

\item Average the decoded matrices in the master node:
\newline
$S=\frac{1}{n} \sum_{i=1}^n \bar M_i$.
\vspace{-2mm}
\item Encode the average in the master node and decode in each node $i$ using its local information 
\vspace{-2mm}
\begin{equation*}
\resizebox{0.9\hsize}{!}{
   $\bar M=\phi^{-1} \left(Q(\phi(S),\phi(M_i), \sqrt{d} \left( \frac{\lambda_{min}(M)}{16 \kappa_{\ell}}+2 n \lambda_{max}(M) \right),\frac{\lambda_{min}(M)}{16 \sqrt{2} \kappa_{\ell}}) \right)$}. 
\end{equation*}
\vspace{-8mm}

\subsection*{(B) Starting Point and Parameters for Descent Direction Quantization:}
\item Choose $D>0$ and $x^{(0)} \in \mathbb{R}^d$, such that 
\begin{equation*}
   \max_i \lbrace \|x^{(0)}-x^*\| , \|x^{(0)}-x_i^*\| \rbrace \leq D.
\end{equation*}
 
\vspace{-2mm}
\item Define the parameters
\vspace{-2mm}
\begin{align*}
   & \xi:=1-\frac{1}{2 \kappa_{\ell}}, K:=\frac{2}{\xi}, \delta:=\frac{\xi(1-\xi)}{4},\\
   &  R^{(t)}:= \frac{\gamma_{\ell}}{2} K \left(1-\frac{1}{4 \kappa_{\ell}} \right)^t D.
\end{align*}
\vspace{-8mm}
\subsection*{(C) Quantizing the Initial Descent Direction:}
\item Compute $\bar M^{-1} \nabla f_i(x^{(0)})$ in each node.
\vspace{-2mm}
\item Encode $\bar M^{-1} \nabla f_i(x^{(0)})$ in each node and decode it in the master node using its local information:
\vspace{-2mm}
\begin{equation*}
    \resizebox{0.9 \hsize}{!}{$v_i^{(0)}=Q\left(\bar M^{-1} \nabla f_i(x^{(0)}),\bar M^{-1} \nabla f_{i_0}(x^{0}),4 n \kappa(M) R^{(0)},\frac{\delta R^{(0)}}{2} \right)$}.
\end{equation*}
\vspace{-8mm}
\item Average the quantized local information in the master node:
\newline
$r^{(0)}=\frac{1}{n} \sum_{i=1}^n v_i^{(0)}$.
\vspace{-2mm}
\item Encode $r^{(0)}$ in the master node and decode it in each machine $i$ using its local information:
\vspace{-2mm}
\begin{equation*}
   \resizebox{0.9 \hsize}{!}{$v^{(0)}=Q \left(r^{(0)},\bar M^{-1} \nabla f_i(x^{(0)}),\left(\frac{\delta}{2}+4 n \kappa(M) \right) R^{(0)},\frac{\delta R^{(0)}}{2} \right)$} .
\end{equation*}
\textbf{For} $t \geq 0$: 
\item Compute 
\vspace{-2mm}
\begin{equation*}
    x^{(t+1)}=x^{(t)}- \eta v^{(t)} 
\end{equation*}
\vspace{-2mm}
for $\eta>0$.
\subsection*{(D) Descent Direction Quantization for Next Steps:}
\item Encode $\bar M^{-1} \nabla f_i(x^{(t)})$ in each node $i$ and decode in the master node using the previous local estimate:
\vspace{-2mm}
\begin{equation*}
    \resizebox{0.9 \hsize}{!}{$v_i^{(t+1)}=Q \left(\bar M^{-1} \nabla f_i(x^{(t+1)} \right), v_i^{(t)},4 n \kappa(M) R^{(t+1)}, \frac{\delta R^{(t+1)}}{2})$}.
\end{equation*}
\vspace{-8mm}
\item Average the quantized local information:
\newline
$r^{(t+1)}=\frac{1}{n} \sum_{i=1}^n v_i^{(t+1)}$.
\vspace{-2mm}
\item Encode $r^{(t+1)}$ in the master node and decode it in each node using the previous global estimate:
\begin{equation*}
    \resizebox{0.9 \hsize}{!}{$v^{(t+1)}=Q \left(r^{(t+1)},v^{(t)}, \left(\frac{\delta}{2}+4 n \kappa(M) \right) R^{(t+1)}, \frac{\delta R^{(t+1)}}{2} \right)$}.
\end{equation*}
\end{enumerate}
\vspace{-4mm}
\noindent\rule[0.5ex]{\linewidth}{1pt}

We now discuss the algorithm's assumptions. First, we assume that an over-approximation $D$ for the distance of the initialization from the minimizer is known. This is practical, especially in the case of GLMs: since the loss functions $\ell_i$ are often quadratics, we can use strong convexity and write
\vspace{-1mm}
\begin{equation*}
    \|x^{(0)}-x^*\|^2 \leq \frac{2}{\mu} (f(x^{(0)})-f^*) \leq \frac{2}{\mu} f(x^{(0)})=:D^2.
    \vspace{-1mm}
\end{equation*}
and similarly for $\|x^{(0)}-x_i^*\|^2$. Further, following \citet{magnusson2020maintaining} (Assumption 2, page 5), the value $f(x^{(0)})$ is often available, for example in the case of logistic regression. Of course, if we are restricted in a compact domain as is the case of \cite{4048825} and \cite{alistarh2020improved}, then the domain itself provides an over approximation for all the distances inside it.
\newline
The parameters $\lambda_{max}(M), \lambda_{min}(M)$ used for quantization of the matrix $M$ are usually assumed to be known. Specifically, it is common in distributed optimization to assume that all nodes know estimates of the smoothness and strong convexity constants of each of the local cost functions \cite{4048825}. In our case this would imply knowing all $\lambda_{max}(M_i),\lambda_{min}(M_i)$. However, we assume knowledge of just $\lambda_{max}(M)$ and $\lambda_{min}(M)$. This also explains the appearance of the extra $\log n$ factor in our GLM bounds, relative to those for Newton's method.
\newline
The convergence and communication complexity of our algorithm are described in the following theorem:
\begin{tcolorbox}
\begin{restatable}{theorem}{convergenceGLM}
\label{thm:convergence_GLM}
The iterates $x^{(t)}$ produced by the previous algorithm with $\eta=\frac{2}{\mu_{\ell}+\gamma_{\ell}}$ satisfy
\begin{align*}
    \| x^{(t)}-x^* \| \leq \left(1-\frac{1}{4\kappa_{\ell}}\right)^t D
    \end{align*}
    and the total number of bits used for communication until $f(x^{(t)})-f^* \leq \epsilon$ is
    \begin{align}
    \label{eq:communication_glm}
    \begin{split}
        &\bigO \left(n d^2 \log \left(\sqrt{d} n \kappa_{\ell} \kappa(M) \right) \right) + \\ & \bigO \left(n d \kappa_{\ell} \log (n \kappa_{\ell} \kappa(M)) \log \frac{\gamma D^2}{\epsilon}\right).
        \end{split}
    \end{align}
\end{restatable}
\end{tcolorbox}

When the accuracy $\epsilon$ is sufficiently small (which is often the case in practice), the first summand is negligible and the total number of bits until reaching it is just
\vspace{-1mm}
\begin{equation*}
        b=\bigO \left(n d \kappa_{\ell} \log (n \kappa_{\ell} \kappa(M)) \log \frac{\gamma D^2}{\epsilon}\right)
        \vspace{-1mm}
\end{equation*}
which gains over quantized gradient descent in \cite{alistarh2020improved} the linear dependence on the condition number of $M$. We prove Theorem \ref{thm:convergence_GLM} in Appendix \ref{app:convergence_GLM}.

%% file: 05_quantized_newton.tex
\section{Quantized Newton's method}

After warming-up with quantizing fixed preconditioners in the case of Generalized Linear Models, we move forward to quantize non-fixed ones. The extreme case of a preconditioner is the whole Hessian matrix; preconditioning with it yields Newton's method, which is computationally expensive, but removes completely the dependency on the condition number from the iteration complexity. We develop a quantized version of Newton's method in order to address a question raised by \cite{alistarh2020improved} regarding whether the communication complexity of minimizing a sum of smooth and strongly convex functions depends linearly on the condition number of the problem. The main technical challenge towards that, is keeping track of the concentration of the Hessians around the Hessian evaluated at the optimum, while the algorithm converges. We show that the linear dependence of the communication cost on the condition number of the problem is not necessary, in exchange with extra dependence on the dimension of the problem, i.e. $d^2$ instead of $d$. This can give significant advantage for low-dimensional and ill-conditioned problems (training generalized linear models is among them). 
\newline
As it is natural for Newton's method, we make the following assumptions for the objective function $f$:
\begin{Assumption}
\label{ass:local_cost_newton}
The functions $f_i$ are all $\gamma$-smooth and $\mu$-strongly convex with a $\sigma$-Lipschitz Hessian, $\gamma,\mu,\sigma>0$.
\end{Assumption}
We note that the lower bound derived by \citet{alistarh2020improved} is obtained for the case that $f_i$ are quadratic functions; quadratic functions indeed satisfy Assumption \ref{ass:local_cost_newton}.
As in the case of GLMs, we define the condition number of the problem to be
\vspace{-1mm}
\begin{equation*}
    \kappa:=\frac{\gamma}{\mu}.
    \vspace{-1mm}
\end{equation*}
We also introduce a constant $\alpha \in [0,1)$, to be specified later, which will control the convergence of the algorithm.

\subsection{Algorithm Description}

We now describe our quantized Newton's algorithm. 
Again, we split the presentation into several parts: local initialization (A), estimating the initial Hessian modulo quantization (B), as well as the quantized initial descent direction (C), and finally, quantization and update for each iteration (D,E).

\noindent\rule[0.5ex]{\linewidth}{1pt}
\vspace{-8mm}
\begin{enumerate}
\item Choose the master node at random, e.g. $i_0$.
\subsection*{(A) Starting Point and Parameters for Hessian Quantization:}
\item Choose $x^{(0)} \in \mathbb{R}^d$, such that 
\begin{equation*}
    \max_i \lbrace\|x^{(0)}-x^*\|,\|x^{(0)}-x_i^*\| \rbrace \leq \frac{\alpha \mu}{2 \sigma}.
\end{equation*}
\vspace{-8mm}
\item We define the parameter
\begin{align*}
    G^{(t)}=\frac{\mu}{4} \alpha \left(\frac{1+\alpha}{2} \right)^t.
\end{align*}
\subsection*{(B) Initial Hessian Quantized Estimation:}
\item Compute $\nabla^2 f_i(x^{(0)})$ in each node.
\vspace{-2mm}
\item Encode $\nabla^2 f_i(x^{(0)})$ in each node $i$ and decode it in the master node $i_0$ using its information:
\vspace{-2mm}
\begin{align*}
\resizebox{0.9 \hsize}{!}{$H_0^i= \phi^{-1}\left(Q \left(\phi(\nabla^2 f_i(x^{(0)})), \phi(\nabla^2 f_{i_0}(x^{(0)})), 2 \sqrt{d} \gamma, \frac{G^{(0)}}{2 \sqrt{2} \kappa} \right)\right)$}.
\end{align*}
\vspace{-8mm}
\item Average the decoded matrices in the master node:
\newline
$S_0=\frac{1}{n} \sum_{i=1}^n H_0^i$.
\vspace{-2mm}
\item Encode the average in the master node and decode in each node $i$ using its local information 
\vspace{-2mm}
\begin{equation*}
\resizebox{0.9 \hsize}{!}{
   $H_0= \phi^{-1}\left(Q \left(\phi(S_0), \phi(\nabla^2 f_i(x^{(0)})),  \sqrt{d} \left( \frac{G^{(0)}}{2\kappa} + 2 \gamma \right) , \frac{G^{(0)}}{2 \sqrt{2} \kappa} \right)\right)$}. 
\end{equation*}
\vspace{-8mm}
\subsection*{Parameters for Descent Direction Quantzation:}
\item Define the parameters
\vspace{-2mm}
\begin{align*}
    &\theta:=\frac{\alpha(1-\alpha)}{4},
     K:=\frac{2}{\alpha},
     P^{(t)}:= \frac{\mu}{2 \sigma} K \alpha \left( \frac{1+\alpha}{2}\right)^t.
\end{align*}
\vspace{-8mm}
\subsection*{(C) Initial Descent Direction Quantized Estimation:}
\item Compute $H_0^{-1} \nabla f_i(x^{(0)})$ in each node.
\vspace{-2mm}
\item Encode $H_0^{-1} \nabla f_i(x^{(0)})$ in each node and decode it in the master node using its local information:
\vspace{-2mm}
\begin{equation*}
    \resizebox{0.9 \hsize}{!}{$v_i^{(0)}=Q \left(H_0^{-1} \nabla f_i(x^{(0)}), H_0^{-1} \nabla f_{i_0} (x^{(0)}),4 \kappa P^{(0)},\frac{\theta P^{(0)}}{2} \right)$}.
\end{equation*}
\vspace{-8mm}
\item Average the quantized local information:
\newline
$p^{(0)}=\frac{1}{n} \sum_{i=1}^n v_i^{(0)}$.
\vspace{-2mm}
\item Encode $p^{(0)}$ in the master node and decode it in each machine $i$ using its local information:
\vspace{-2mm}
\begin{equation*}
   \resizebox{0.9 \hsize}{!}{$v^{(0)}=Q \left(p^{(0)}, H_0^{-1} \nabla f_i (x^{(0)}),\left(\frac{\theta}{2}+4 \kappa \right) P^{(0)}, \frac{\theta P^{(0)}}{2} \right)$} .
\end{equation*}
\vspace{-4mm}
\newline
\textbf{For} $t \geq 0$: 
\item Compute 
\vspace{-2mm}
\begin{equation*}
    x^{(t+1)}=x^{(t)}-v^{(t)}. 
\end{equation*}
\vspace{-8mm}
\subsection*{(D) Hessian Quantized Estimation for Next Steps:}
\item Compute $\nabla^2 f_i(x^{(t+1)})$ in each node.
\vspace{-2mm}
\item Encode $\nabla^2 f_i(x^{(t+1)})$ in each node $i$ and decode in the master node using the previous local estimate:
\vspace{-2mm}
\begin{equation*}
    \resizebox{0.9 \hsize}{!}{$H_{t+1}^i=\phi^{-1}\left(Q \left(\phi(\nabla^2 f_i(x^{(t+1)})), \phi(H_t^i), \frac{10 \sqrt{d}}{1+\alpha} G^{(t+1)}, \frac{ G^{(t+1)}}{2 \sqrt{2} \kappa} \right)\right)$}.
\end{equation*}
\vspace{-8mm}
\item Average the quantized local Hessian information:
\newline
$S_{t+1}=\frac{1}{n} \sum_{i=1}^{n} H_{t+1}^i$.
\vspace{-2mm}
\item Encode $S_{t+1}$ in the master node and decode it back in each node using the previous global estimate:
\vspace{-2mm}
\begin{equation*}
    \resizebox{0.9 \hsize}{!}{$H_{t+1}=\phi^{-1} \left(Q \left(\phi(S_{t+1}), \phi(H_t) , \sqrt{d} \left( \frac{1}{2 \kappa}+ \frac{10}{1+\alpha} \right) G^{(t+1)} , \frac{G^{(t+1)}}{2 \sqrt{2} \kappa} \right) \right)$}.
\end{equation*}
\vspace{-8mm}
\subsection*{(E) Descent Direction Quantized Estimation:}
\item Compute $H_{t+1}^{-1} \nabla f_i(x^{(t+1)})$ in each node.
\vspace{-2mm}
\item Encode $H_{t+1}^{-1} \nabla f_i(x^{(t+1)})$ in each node $i$ and decode in the master node using the previous local estimate:
\vspace{-2mm}
\begin{equation*}
    \resizebox{0.9 \hsize}{!}{$v^{(t+1)}_i=Q \left(H_{t+1}^{-1} \nabla f_i(x^{(t+1)}), v^{(t)}_i, 11 \kappa P^{(t+1)} , \frac{\theta P^{(t+1)}}{2} \right)$}.
\end{equation*}
\vspace{-8mm}
\item Average the quantized local Hessian information:
\newline
$p^{(t+1)}=\frac{1}{n} \sum_{i=1}^{n} v^{(t+1)}_i$.
\vspace{-2mm}
\item Encode $S_{t+1}$ in the master node and decode it back in each node using the previous global estimate:
\vspace{-2mm}
\begin{equation*}
    \resizebox{0.9 \hsize}{!}{$v^{(t+1)}=Q \left(p^{(t+1)}, v^{(t)},\left(\frac{\theta}{2}+11 \kappa \right) P^{(t+1)}, \frac{\theta P^{(t+1)}}{2} \right)$}.
\end{equation*}
\end{enumerate}
\vspace{-4mm}
\noindent\rule[0.5ex]{\linewidth}{1pt}
The restriction of the initialization $x^{(0)}$ is standard for Newton's method, which is known to converge only \textit{locally}. Usually $x^{(0)}$ is chosen such that $\alpha \geq \frac{\sigma}{\mu} \|x^{(0)}-x^*\|$, while we choose it such that $\alpha \geq 2 \frac{\sigma}{\mu} \|x^{(0)}-x^*\|$ (and the same for $x_i^*$ in the place of $x^*$). This difference occurs from the extra errors due to quantization. This assumption implies also that the minima of the local costs cannot be too far away from each other.
\newline
We now state our theorem on communication complexity of quantized Newton's algorithm, which is the main result of the paper. The proof is in Appendix \ref{app:quant_newton}, and relies on analyzing the behaviour of both the quantized Hessian estimates and the quantized descent direction estimates simultaneously, as can be seen in Lemma \ref{le:desc_direction_newton}.
\begin{tcolorbox}
\begin{restatable}{theorem}{maintheorem}
\label{thm:main_theorem}
The iterates of the quantized Newton's method starting from a point $x^{(0)}$, such that
\begin{equation*}
    \|x^{(0)}-x^*\| \leq \frac{\mu}{4 \sigma}  \left(\alpha=\frac{1}{2}\right)
\end{equation*}
satisfy
\begin{equation*}
    \|x^{(t)}-x^*\| \leq \frac{\mu}{4 \sigma} \left( \frac{3}{4} \right)^t 
\end{equation*}
and the communication cost until reaching accuracy $\epsilon$ in terms of function values is
\begin{equation}
\label{eq:communication_newton}
   \bigO\left(nd^2 \log\left( \sqrt{d} \kappa \right) \log \frac{\gamma \mu^2}{\sigma^2 \epsilon} \right)
\end{equation}
many bits in total.
\end{restatable}
\end{tcolorbox}
We note that the lower bound derived in \cite{alistarh2020improved} is for the case that all functions $f_i$ are quadratics. For quadratics, the Hessian is constant, thus $\sigma=0$ and $\alpha$ can be chosen equal to $0$ as well. Then, (non-distributed) Newton's method converges in only one step. However, in the distributed case, $\sigma = 0$ implies $G^{(t)}=0$, thus the estimation of $\nabla^2 f(x^{(t)})$ must be exact. This would mean that we need to use an infinite number of bits, and this can be seen also in our communication complexity results. In order to apply our result in a practical manner, we need to allow the possibility for strictly positive quantization error of the Hessian, thus we must choose $\sigma>0$.

%% file: 06_function_estimation.tex
\section{Estimation of the Minimum in the Master}
In the previous sections we computed an approximated minimizer of our objective function up to some accuracy and counted the communication cost of the whole process. We now extend our interest to the slightly harder problem of estimating the minimum $f^*$ of the function $f$ (which is again assumed to be $\gamma$-smooth and $\mu$-strongly convex) in the master node with accuracy $\epsilon$. This extension is not considered in \cite{magnusson2020maintaining}, but is discussed in \cite{alistarh2020improved}. To that end, we estimate the minimizer $x^*$ of $f$ by a vector $x^{(t)}$, such that $f(x^{(t)}) -f^* \leq \frac{\epsilon}{2}$, and the communication cost of doing that is again given by expression (\ref{eq:communication_glm}) for GLM training and expression (\ref{eq:communication_newton}) for Newton's method.
\newline
We denote $x_i^*$ the minimizer of the local cost function $f_i$ and $f_i^*:=f_i(x_i^*)$ its minimum. We also assume that we are aware of an over approximation $C>0$ of the maximum distance of $x^*$ from the minimizers of the local costs $x_i^*$, i.e.
    $\max_{i=1,...,n} \|x^*-x_i^*\| \leq C$
and a radius $c>0$ for the minima of the local costs:
    $\max_{i=1,...,n} \mid f_i^* \mid \leq c$.
Estimating these constants can be feasible in many practical situations: 
\vspace{-3mm}
\begin{itemize}
    \item We can always bound the quantity $\max_{i=1,...,n} \|x^*-x_i^*\| $ by a known constant if we set our problem in a compact domain as it is the case in \cite{4048825} and \cite{alistarh2020improved}. Also, if our local data are obtained from the same distribution, then we do not expect the minimizers of the local costs to be too far away from the global minimizer.
    \vspace{-2mm}
    \item The minima $f_i^*$ of the local costs are often exactly $0$ (as assumed in \cite{alistarh2020improved}). This is because the local cost functions $f_i$ are often quadratics, as it happens in the case of GLMs. In the worst case, knowing just that $f_i \geq 0$, we can write
    \vspace{-1mm}
    \begin{equation*}
       \mid f_i^* \mid = f_i^* \leq f_i(x^{(0)}) \leq n f(x^{(0)})
       \vspace{-1mm}
    \end{equation*}
    and the value $f(x^{(0)})$ is often available as discussed in Section \ref{sec:GLMs} and in \cite{magnusson2020maintaining}.
\end{itemize}

For estimating the minimum $f^*$, we start by computing $f_i(x^{(t)})$ in each node $i$ and communicate them to the master node $i_0$ as follows:
\begin{equation*}
    q_i^{(t)}:=Q(f_i(x^{(t)}), f_{i_0}(x^{(t)}),2  (\gamma C^2+c), \epsilon/2).
\end{equation*}
Then the master node computes and outputs the average 
\begin{equation*}
    \bar f= \frac{1}{n} \sum_{i=1}^n q_i^{(t)}.
\end{equation*}
\begin{restatable}{Proposition}{functionvalue}
The value $\bar f$ which occurs from the previous quantization procedure is an estimate of the true minimum $f^*$ of $f$ with accuracy $\epsilon$ and the cost of quantization is
\begin{equation*}
    \bigO \left(n \log \frac{\gamma C^2+c}{\epsilon} \right).
\end{equation*}
if $\epsilon$ is sufficiently small.
\end{restatable}
The proof is presented in Appendix \ref{app:function_value}.
\newline
Thus, for the problem that the master node needs to output estimates for both the minimizer and the minimum with accuracy $\epsilon$ in terms of function values, the total communication cost is at most
\begin{align*}
\bigO \left(n d \kappa_{\ell} \log (n \kappa_{\ell} \kappa(M)) \log \frac{\gamma( C^2 + D^2)+c) }{\epsilon}\right)
\end{align*}
many bits in total for QPGD-GLM
\begin{align*}
    \bigO\left(nd^2 \log\left( \sqrt{d} \kappa \right) \log \left( \left( \gamma \left(\frac{ \mu^2}{\sigma^2}+ C^2 \right)+c \right) \frac{1}{\epsilon} \right)  \right).
\end{align*}
many bits in total for quantized Newton's method when $\epsilon$ is sufficiently small.

%% file: 07_experiments.tex
\section{Experiments}
\label{sec:experiments}

\hide{
\begin{figure*}[h]
\begin{center}
\begin{minipage}[b]{0.48\linewidth}
 \centering
  \includegraphics[width=\columnwidth]{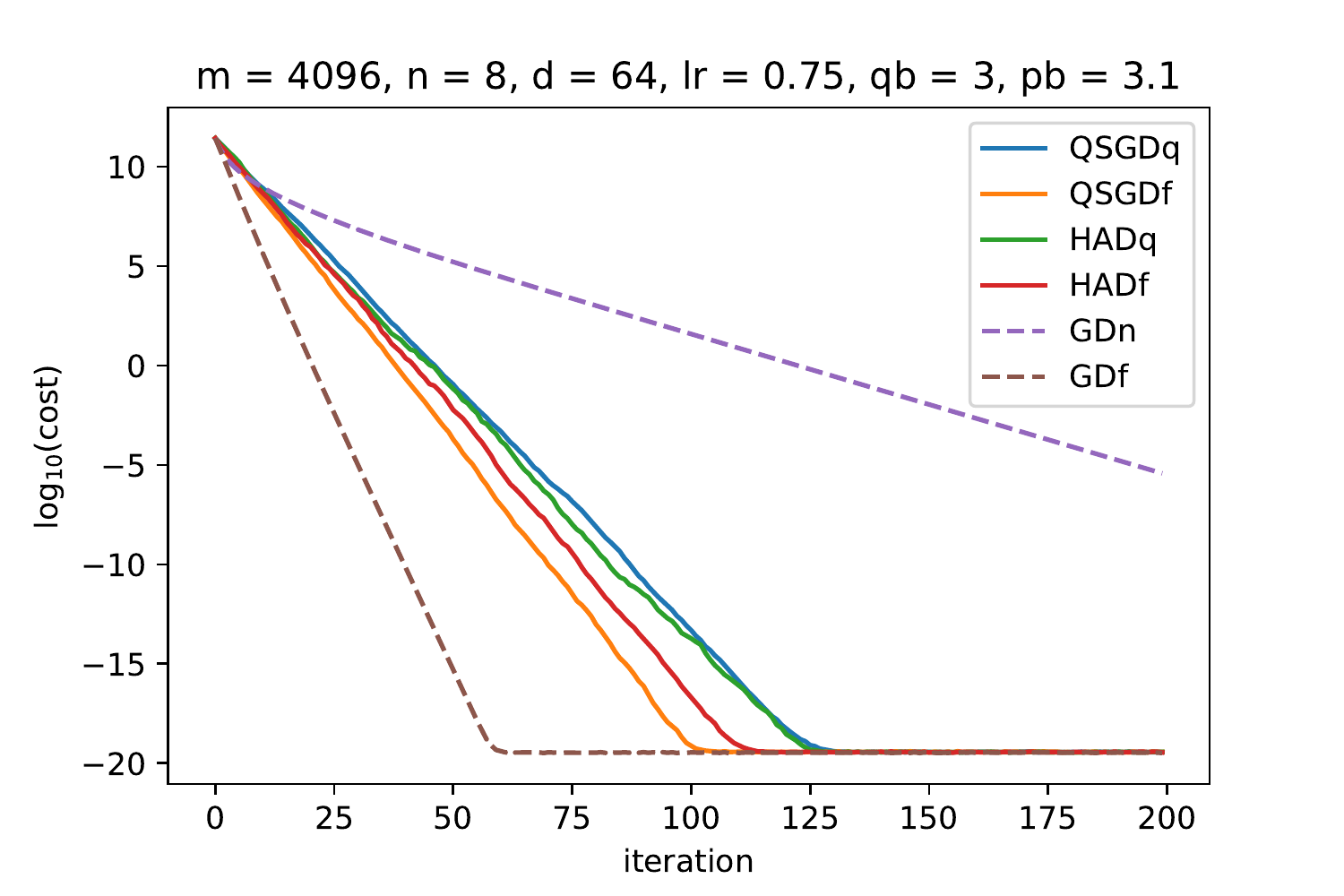}
\subcaption{Performance on synthetic data}
\label{fig:synthetic}
  \end{minipage}
    \quad
    \begin{minipage}[b]{0.48\linewidth}
        \centering
        \includegraphics[width=\columnwidth]{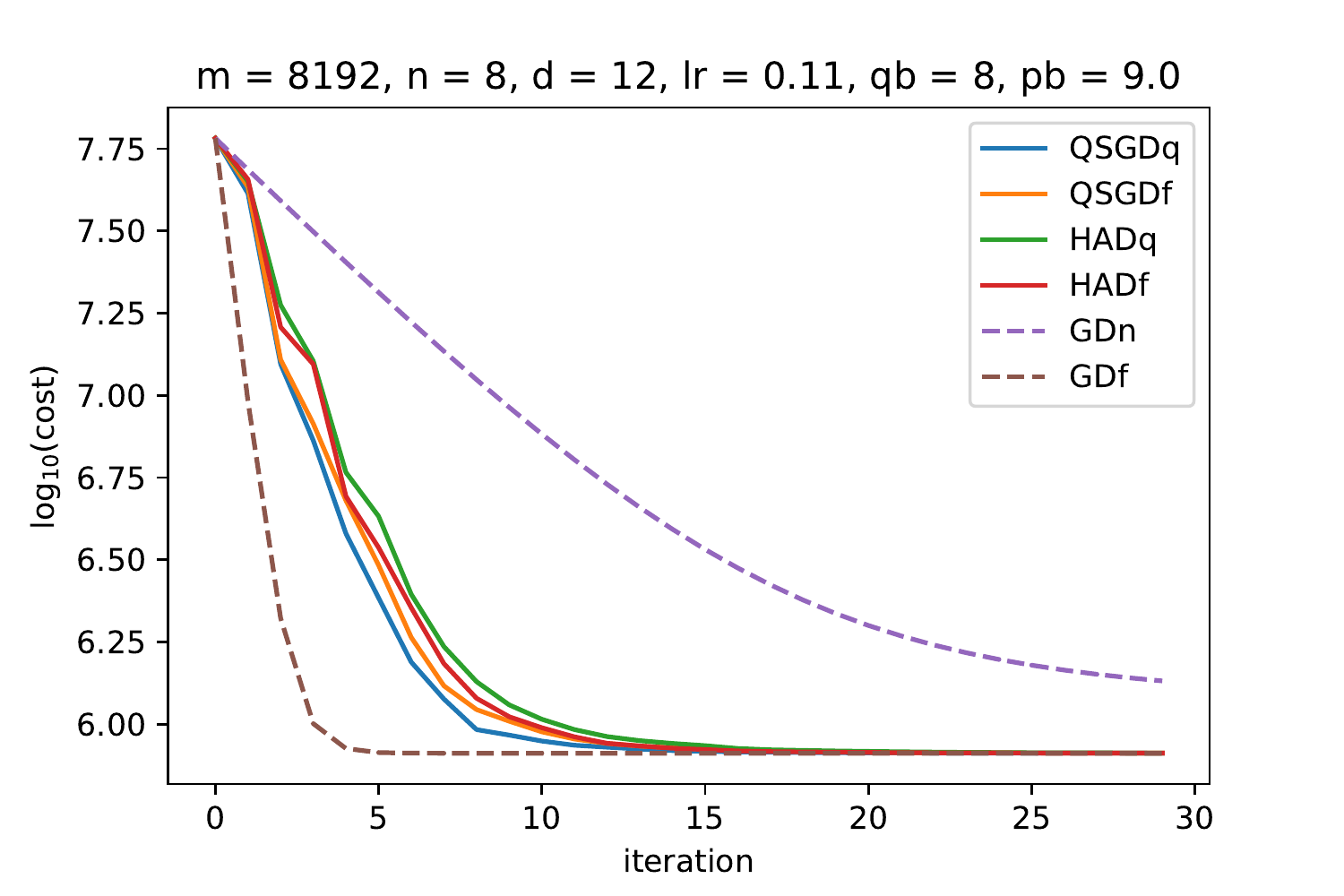}
\subcaption{Performance on \texttt{cpusmall\_scale}}
\label{fig:cpusmall}
    \end{minipage}
\end{center}
\vskip -0.2in
\end{figure*}

\begin{figure*}[h]
\begin{center}
\begin{minipage}[b]{0.45\linewidth}
 \centering
  \includegraphics[width=\columnwidth]{Figures/german_numer.txt_lr_0.005_n_5_qb_82.pdf}
\subcaption{Logistic Regression (placeholder)}
\label{fig:synthetic}
  \end{minipage}
    \quad
    \begin{minipage}[b]{0.45\linewidth}
        \centering
  \includegraphics[width=\columnwidth]{Figures/german_numer.txt_lr_0.005_n_5_qb_82.pdf}
\subcaption{Logistic Regression (placeholder)}
\label{fig:cpusmall}
    \end{minipage}
\end{center}
\vskip -0.2in
\end{figure*}

    \begin{figure*}
        \centering
        \begin{subfigure}[b]{0.45\textwidth}
            \centering
            \includegraphics[width=\textwidth]{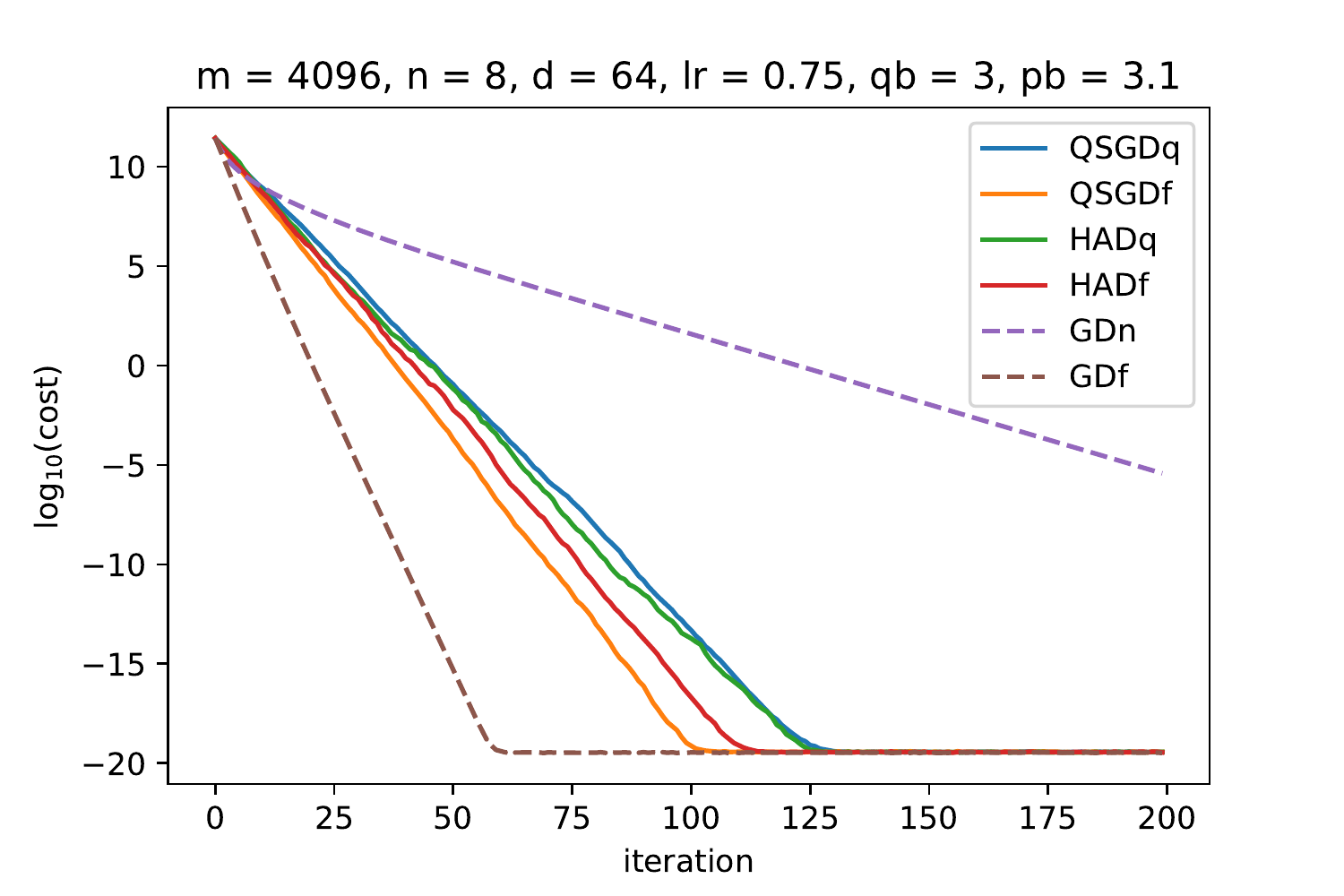}
            \caption[Network2]%
            {{\small Network 1}}    
            \label{fig:mean and std of net14}
        \end{subfigure}
        \hfill
        \begin{subfigure}[b]{0.45\textwidth}  
            \centering 
            \includegraphics[width=1.1\textwidth]{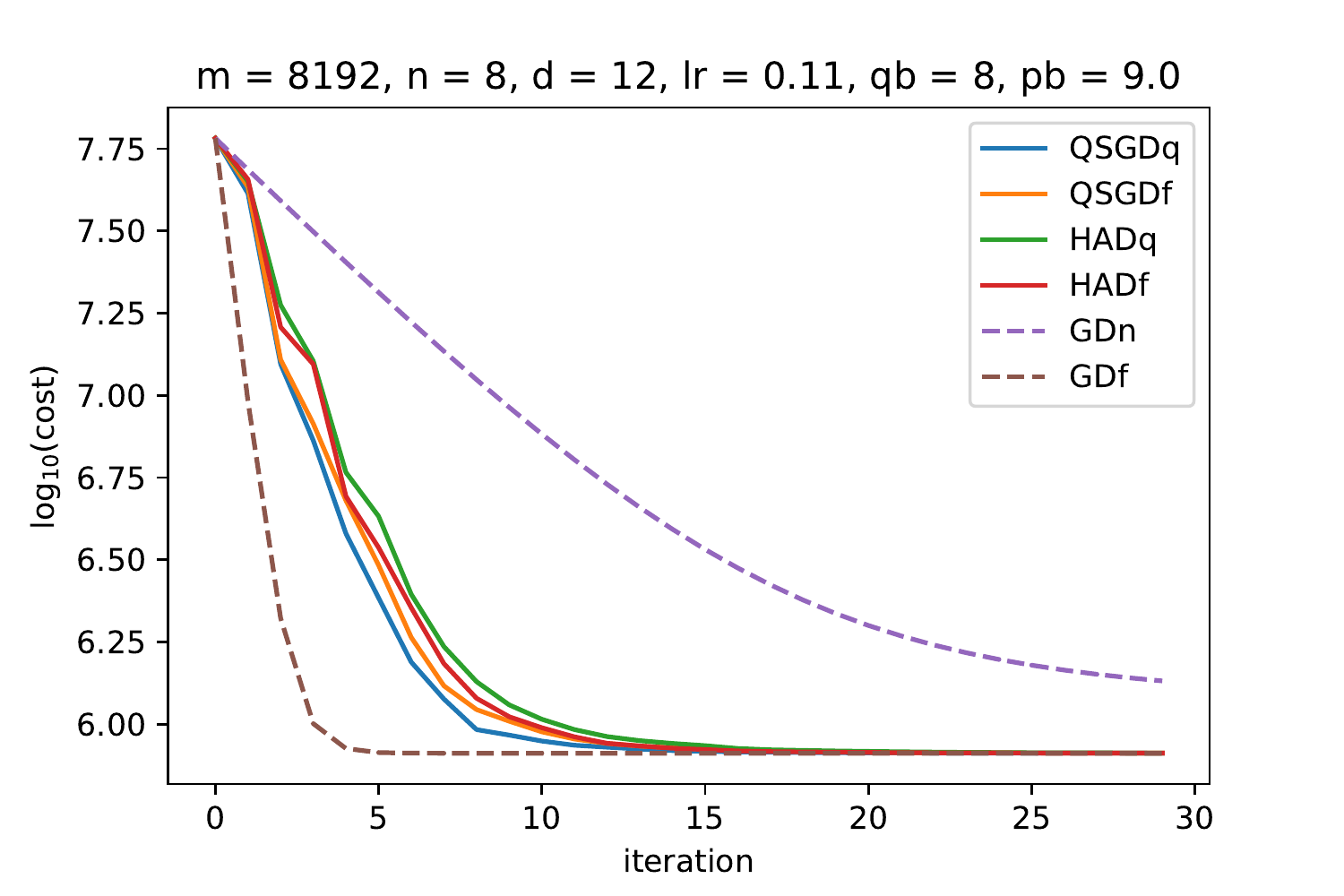}
            \caption[]%
            {{\small Network 2}}    
            \label{fig:mean and std of net24}
        \end{subfigure}
        \vskip\baselineskip
        \begin{subfigure}[b]{0.4\textwidth}   
            \centering 
            \includegraphics[width=\textwidth]{Figures/german_numer.txt_lr_0.005_n_5_qb_82.pdf}
            \caption[]%
            {{\small Network 3}}    
            \label{fig:mean and std of net34}
        \end{subfigure}
        \hfill
        \begin{subfigure}[b]{0.4\textwidth}   
            \centering 
            \includegraphics[width=\textwidth]{Figures/german_numer.txt_lr_0.005_n_5_qb_82.pdf}
            \caption[]%
            {{\small Network 4}}    
            \label{fig:mean and std of net44}
        \end{subfigure}
        \caption[ The average and standard deviation of critical parameters ]
        {\small The average and standard deviation of critical parameters: Region R4} 
        \label{fig:mean and std of nets}
    \end{figure*}
 }
 
\begin{figure*}[h]
\begin{center}
\begin{minipage}[b]{0.33\linewidth}
 \centering
\includegraphics[width=\textwidth]{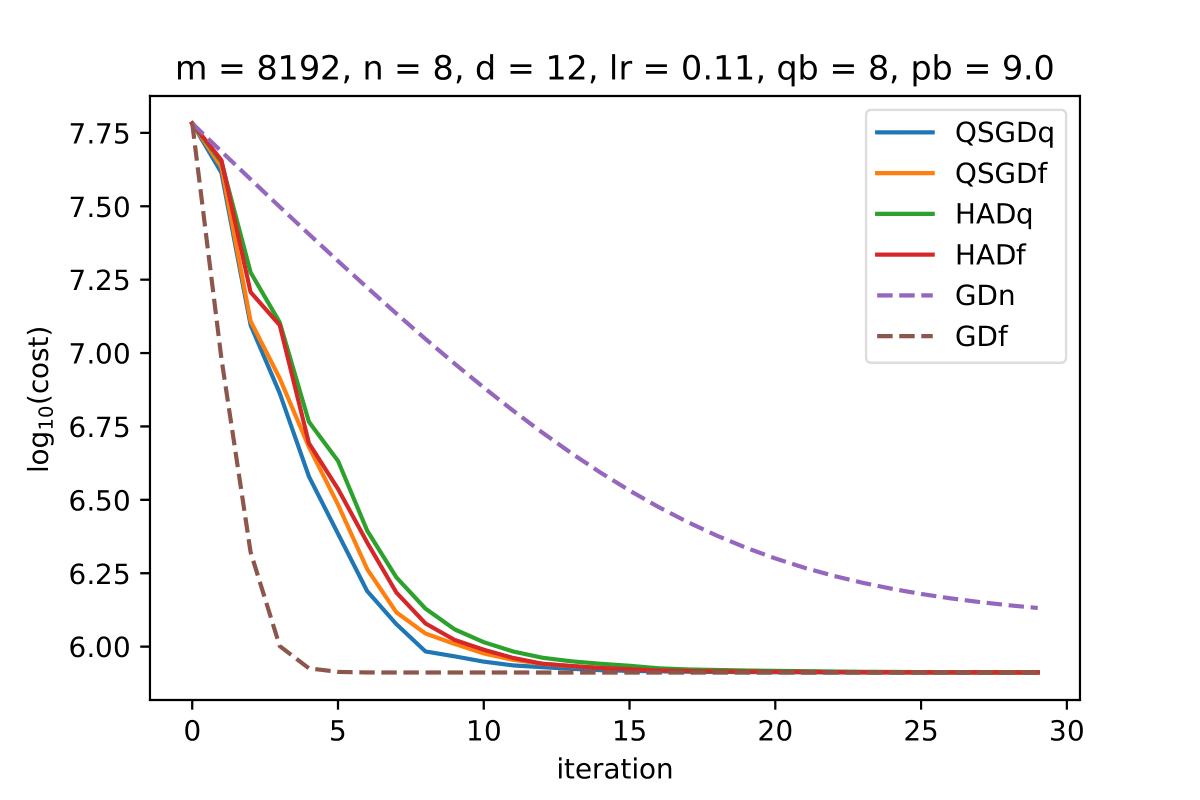}
\subcaption{Least-squares regression performance on \texttt{cpusmall\_scale}}
\label{fig:cpusmall}
  \end{minipage}
    \quad
    \begin{minipage}[b]{0.305\linewidth}
        \centering
        \includegraphics[width=\textwidth]{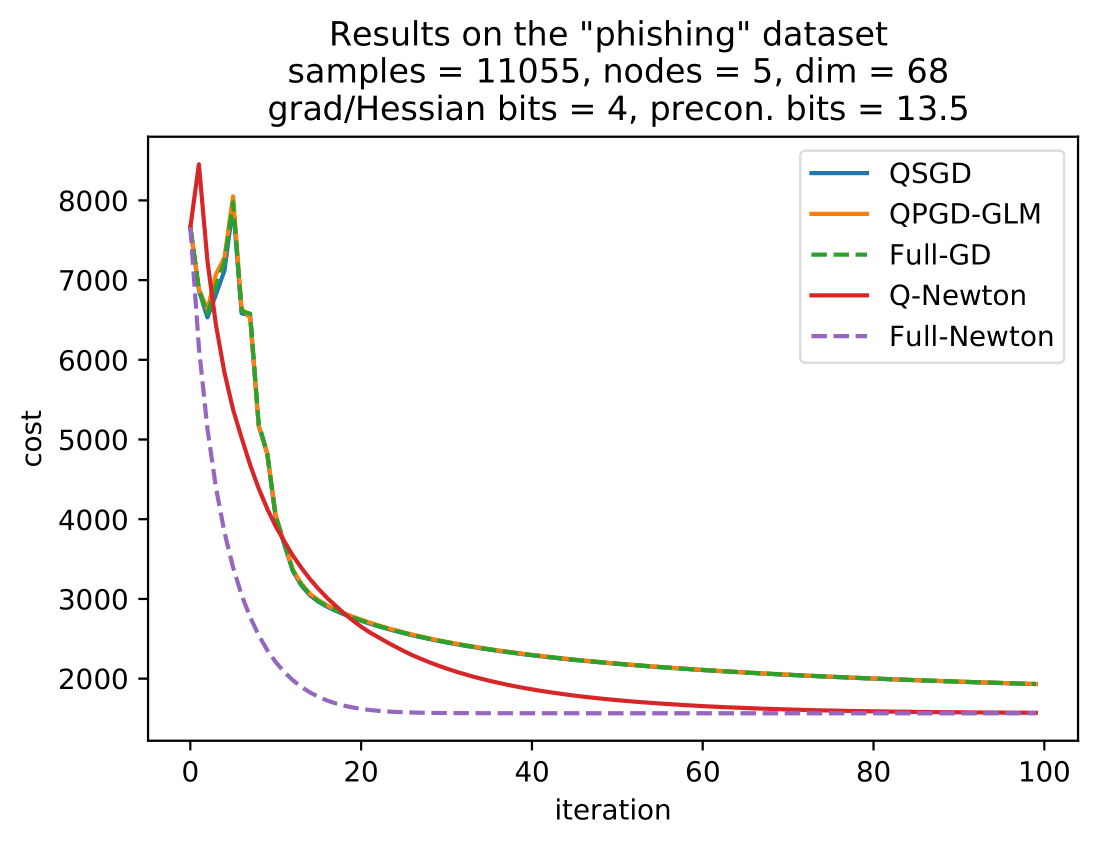}
\subcaption{Logistic regression performance on \texttt{phishing} }
\label{fig:phishing}
\end{minipage}
        \quad
\begin{minipage}[b]{0.305\linewidth}
\centering
\includegraphics[width=\textwidth]{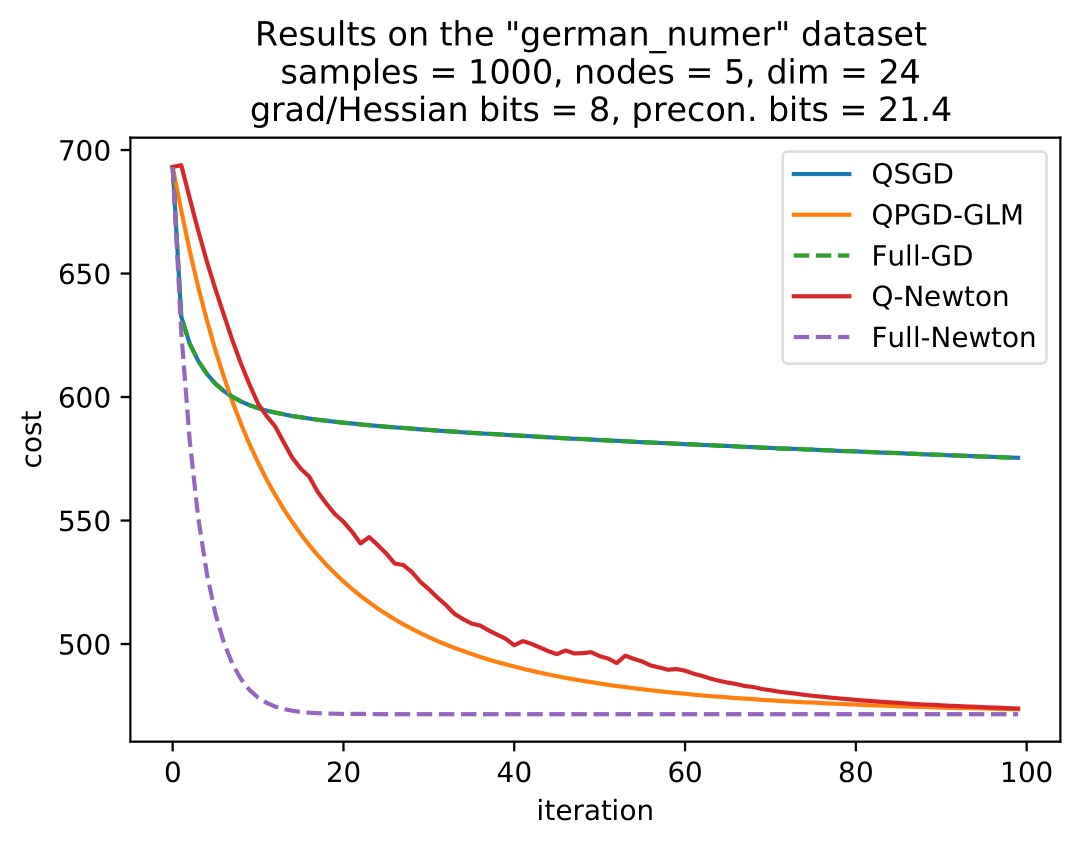}
\subcaption{Logistic regression performance on \texttt{german\_numer}}
\label{fig:german_numer}
    \end{minipage}
\end{center}
\vskip -0.2in
\end{figure*}

\subsection{Experiment 1: Least-Squares Regression}
We first test our method experimentally to compress a parallel solver for least-squares regression. The setting is as follows: we are given as input a data matrix $A$, with rows randomly partitioned evenly among the nodes, and a target vector $b$, with the goal of finding $ x^* = \text{argmin}_{ x}\|A x- b\|^2_2$. Since this loss function $f(x):= \|A x- b\|^2_2$ is quadratic, its Hessian is constant, and so Newton's method and QPGD-GLM are equivalent: in both cases, we need only to provide the preconditioner matrix $A^T A$ in the first iteration, and machines can henceforth use it for preconditioning in every iteration.

To quantize the preconditioner matrix, we apply the `practical version' (that is, using the cubic lattice with $\bmod$-based coloring) of the quantization method of \cite{davies2021new}, employing the `error detection' method in order to adaptively choose the number of bits required for the decoding to succeed. Each node $i$ quantizes the matrix $A_i^T A_i$, which is decoded by the master node $i_0$ using $A_{i_0}^T A_{i_0}$. Node $i_0$ computes the average, quantizes, and returns the result to the other nodes, who decode using  $A_i^T A_i$.

To quantize gradients, we use two leading gradient quantization techniques: QSGD \cite{alistarh2016qsgd}, and the Hadamard-rotation based method of \cite{pmlr-v70-suresh17a}, since these are optimized for such an application.\footnote{There is a wide array of other gradient quantization methods; we use these two as a representative examples, since we are mostly concerned with the effects of preconditioner quantization.} In each iteration (other than the first), we quantize the \emph{difference} between the current local gradient and that of last iteration, average these at the master node $i_0$, and quantize and broadcast the result.

\paragraph{Compared Methods.}
In Figure \ref{fig:cpusmall} we compare the following methods: GDn and GDf are full-precision (i.e., using 32-bit floats) gradient descent using \emph{no preconditioning} and \emph{full-precision preconditioning} respectively, as baselines. QSGDq and QSGDf use QSGD for gradient quantization, and the \emph{quantized} and \emph{full-precision} preconditioner respectively. HADq and HADf are the equivalents using instead the Hadamard-rotation method for gradient quantization. When using a preconditioner, we rescale preconditioned gradients to preserve $\ell_2$-norm, so that our comparison is based only on update direction and not step size.

\paragraph{Parameters.}
In addition to $m$, $n$, and $d$, we also have the following parameters: the learning rate (lr in the figure titles) is set close to the maximum for which gradient descent will converge, since this is the regime in which preconditioning can help. The number of bits per coordinate used to quantize gradients (qb) and preconditioners (pb) are also shown; the latter is an average since the quantization method uses a variable number of bits\footnote{These quantization methods (and most others) also require exchange of two full-precision scalars, which are not included in the per-coordinate costs since they are independent of dimension.}. The results presented are an average of the cost function per descent iteration, over $10$ repetitions with different random seeds.

\hide{
\paragraph{Synthetic Data.}
We first apply the methods to synthetic data: our data matrix $A$ consists of independent Gaussian entries with variance $1$, we choose our target optimum $x^*$ to be Gaussian with variance $1000$, and set $B=Ax^*$.

Figure \ref{fig:synthetic} shows that, even using only $\sim 3$ bits per coordinate for both gradients and preconditioner, we achieve signicantly faster convergence than full-precision gradient descent without preconditioning, and converge to essentially as good a solution as full-precision gradient descent with preconditioning.}

\paragraph{Dataset}
We use the dataset \texttt{cpusmall\_scale} from LIBSVM \cite{LibSVM}. Here we outperform non-preconditioned gradient descent and approach the performance of full-precision preconditioned gradient descent using significantly reduced communication (Figure \ref{fig:cpusmall}).

\subsection{Experiment 2: Logistic Regression}
In order to compare the performance of Q-Newton and QPGD-GLM, we implement a common application in which the Hessian is \emph{not} constant: logistic regression, for binary classification problems. 

QPGD-GLM, QSGD, and full-precision gradient descent are implemented as before; we now add full-precision Newton's method for comparison, and our Q-Newton algorithm. The latter uses the quantization method of \cite{davies2021new} for the initial Hessian (as for QPGD-GLM), and QSGD for subsequent Hessian updates.

Rather than re-scaling gradients, we take a different approach to choosing a learning rates in order to compare the methods fairly: we test each with learning rates in $\{2^{-0}, 2^{-1},2^{-2},\dots\}$, and plot the highest rate for which the method stably converges. Our results are averaged over five random seeds.

We demonstrate the methods on the \texttt{phishing} and \texttt{german\_numer} datasets from the LIBSVM collection \cite{LibSVM}, in Figures \ref{fig:phishing} and \ref{fig:german_numer} respectively. The former demonstrates that Q-Newton improves over (even full precision) first-order methods, while quantizing Hessians at only 4 bits per coordinate. The latter demonstrates an instance in which QPGD-GLM is even faster, since it remains stable under a higher learning rate.

%% file: 08_discussion.tex
\section{Discussion}
We proposed communication-efficient versions for two fundamental optimization algorithms, and analyzed their convergence and communication complexity. 
Our work shows that quantizing second-order information can i) theoretically yield to communication complexity upper bounds with sub-linear dependence on the condition number of the problem, and ii) empirically achieve superior performance over vanilla methods.

There are intriguing questions for future work:
\newline
The $\log \kappa$-dependency for Newton's method occurs because of our bounds for the input and output variance of quantization. It would be interesting to see whether this dependency can be avoided, making the bounds completely independent of the condition number.

Another interesting question is whether the $\log d$-dependency can be circumvented. $\log d$ is obtained directly from the use of the vectorization $\phi$ and could be  avoided by quantization using lattices with good spectral norm properties. We are however unaware of such lattice constructions.

One key issue left is the $d^2$-dependence for the generalized Newton's method, which is due to quantization of $d^2$-dimensional preconditioners. 
It would be interesting to determine if linear communication per round can be achieved in the general setting we consider here. 

Finally, we would like to point out that there exist more second order methods with superior guarantees compared to vanilla Newton, such as cubic regularization~\cite{RePEc:cor:louvrp:1927}. A very interesting direction for future work would be to investigate whether it is possible to run these algorithms in a distributed setting with limited communication by adding quantization.

\paragraph{Acknowledgements} The authors would like to thank Janne Korhonen, Aurelien Lucchi, Celestine Mendler-Dünner and Antonio Orvieto for helpful discussions.  FA and DA were supported during this work by the European Research Council (ERC) under the European Union’s Horizon 2020 research and innovation programme (grant agreement No 805223 ScaleML). PD was supported by the European Union’s Horizon 2020 programme under the Marie Skłodowska-Curie grant agreement No. 754411.

%% file: appendix.tex
\newpage
\onecolumn
\appendix

\begin{center}
	\textbf{\Large{Appendix: Proofs and Supplementaries}}
\end{center}

\section{The isomorphism $\phi$}
\label{app:isomorphism}

\basic*
\begin{proof}
The Frobenius norm of a matrix $P=(p_{ij})_{i,j=1}^d$ is defined as
\begin{equation*}
   \| P \|_F= \sqrt{\sum_{i,j=1}^d p_{ij}^2} 
\end{equation*}
thus
\begin{align*}
    \| P-P' \|_F^2 = \sum_{i,j=1}^d (p_{ij}-p_{ij}')^2 = \sum_{i=j} (p_{ij}-p_{ij}')^2+\sum_{i \neq j} (p_{ij}-p_{ij}')^2 =\sum_{i=j} (p_{ij}-p_{ij}')^2+2 \sum_{i<j} (p_{ij}-p_{ij}')^2=:X+2Y
\end{align*}
with $P'=(p'_{ij})_{i,j=1}^d$.
\newline
We also have 
\begin{equation*}
   \| \phi(P)-\phi(P') \|^2= \sum_{i=j} (p_{ij}-p_{ij}')^2+ \sum_{i<j} (p_{ij}-p_{ij}')^2=:X+Y
\end{equation*}
Thus
\begin{equation*}
    \|\phi(P)-\phi(P')\|_2 \leq \|P-P'\|_F \leq \sqrt{2} \|\phi(P)-\phi(P')\|_2.
\end{equation*}
Now for the $\ell_2$ norm, we have
\begin{equation*}
    \| P-P' \|_2 \leq \| P-P' \|_F \leq \sqrt{2} \| \phi(P)-\phi(P') \|_2
\end{equation*}
and
\begin{equation*}
    \| P-P' \|_2 \geq \frac{1}{\sqrt{d}} \| P-P' \|_F \geq \frac{1}{\sqrt{d}} \| \phi(P)-\phi(P') \|_2
\end{equation*}
and the desired result follows.
\end{proof}

\section{Technicalities regarding GLMs}
\label{app:GLM_technicalities}
We firstly compute the Hessian of the global cost function $f$ in terms of the Hessian of the global loss function $\ell$:
\begin{lemma}
We have
\begin{equation*}
    \nabla^2 f(x)=A^T \ell (Ax) A.
\end{equation*}
\end{lemma}
\begin{proof}
We start by computing the gradient of $f$. We fix an arbitrary vector $v \in \mathbb{R}^d$ and we write
\begin{align*}
    &\langle \nabla f(x),v \rangle = d_xf(x)v = d_x(\ell(Ax))v=d_y(\ell(y))|_{y=Ax} d_x(Ax)v = d_y(\ell(y))|_{y=Ax} Av \\ & = \langle \nabla \ell(y)|_{y=Ax}, Av \rangle = (Av)^T \nabla \ell (Ax)=v^T A^T \nabla \ell (Ax)= \langle A^T \nabla \ell (Ax), v \rangle
\end{align*}
Since $v$ is arbitrary, the gradient of $f$ is
\begin{equation*}
    \nabla f(x)=A^T \nabla \ell (Ax)
\end{equation*}
For the Hessian, we have
\begin{align*}
    \nabla_x^2 f(x) = \nabla_x \nabla_x f(x) = \nabla_x (A^T \nabla_x \ell (Ax)) = A^T \nabla_x (\nabla_x \ell (Ax))= A^T \nabla_y (\nabla_y \ell(y))|_{y=Ax} \nabla_x(Ax) = A^T \nabla^2 \ell(Ax) A. 
\end{align*}
\end{proof}

We recall standard technical results from linear algebra in order to prove Proposition \ref{prop:improved_condition_number}. They will be useful also in the proof of Proposition \ref{prop:exact_precond} and Lemma \ref{le:inexact_precond}.

\begin{lemma}
\label{le:eig commute}
Given matrices $P \in \mathbb{R}^{m \times d}$ and $Q \in \mathbb{R}^{d \times m}$, we have that $PQ$ and $QP$ have exactly the same \textbf{non-zero} eigenvalues.
\end{lemma}
 \begin{proof}
 Let $\lambda \neq 0$ an eigenvalues of $PQ$. Then there exists $v \neq 0$, such that $PQv=\lambda v$. Multiplying both sides by $Q$, we get $QP(Qv)=\lambda (Qv)$. We know that $Qv \neq 0$, because then $\lambda$ would be $0$. Thus $\lambda$ is an eigenvalue of $QP$ with eigenvector $Qv$. Thus any non-zero eigenvalue of $PQ$ is also an eigenvalue of $QP$. Switching $P$ and $Q$ in the previous argument implies that any non-zero eigenvalue of $QP$ is also an eigenvalue of $PQ$. Thus, $PQ$ and $QP$ have the same non-zero eigenvalues.
 \end{proof}

\begin{corollary}
\label{le:rank}
Given matrices $P \in \mathbb{R}^{m \times d}$ and $Q \in \mathbb{R}^{d \times m}$, we have that
\begin{equation*}
    rank(PQ)=rank(QP)=\min \lbrace rank(P),rank(Q) \rbrace
\end{equation*}
\end{corollary}

\begin{lemma}
\label{le:eig of product}
Given a symmetric positive semi-definite matrix $S \in \mathbb{R}^{m \times m}$ and a symmetric positive definite $T \in \mathbb{R}^{m \times m}$ with eigenvalues
\begin{equation*}
    \lambda_1(S) \leq ... \leq \lambda_m(S)
\end{equation*}
and 
\begin{equation*}
    \lambda_1(T) \leq ... \leq \lambda_m(T)
\end{equation*}
we have that
\begin{equation*}
    \lambda_k(S) \lambda_1(T) \leq \lambda_k(ST) \leq \lambda_k(S) \lambda_m(T)
\end{equation*}
for any $k=1,...,m$.
\end{lemma}
\begin{proof}
We use the min-max principle for the $k$-th eigenvalue of a matrix $A \in \mathbb{R}^{m \times m}$. This reads
\begin{equation*}
    \lambda_k(A)=\min_{\substack{F\subset \mathbb{R}^M \\ \dim(F)=k}} \left( \max_{x\in F\backslash \{0\}} \frac{(Ax,x)}{(x,x)}\right)
\end{equation*}
We know that $\lambda_k(ST)=\lambda_k(\sqrt{T}S\sqrt{T})$.
Since $T$ is symmetric and positive-definite, its square root $\sqrt{T}$ is also symmetric and positive-definite. Thus, we have
\begin{align*}
    \lambda_k(ST)=\lambda_k(\sqrt{T}S\sqrt{T})=\min_{\substack{F\subset \mathbb{R}^M \\ \dim(F)=k}} \left( \max_{x\in F\backslash \{0\}} \frac{(\sqrt{T}S\sqrt{T}x,x)}{(x,x)}\right)=\min_{\substack{F\subset \mathbb{R}^n \\ \dim(F)=k}} \left( \max_{x\in F\backslash \{0\}} \frac{(S\sqrt{T}x,\sqrt{T}x)}{(\sqrt{T}x,\sqrt{T}x)}
\frac{(Tx,x)}{(x,x)}\right)
\end{align*}
Thus
\begin{equation*}
     \min_{\substack{F\subset \mathbb{R}^n \\ \dim(F)=k}} \left( \max_{x\in F\backslash \{0\}} \frac{(S\sqrt{T}x,\sqrt{T}x)}{(\sqrt{T}x,\sqrt{T}x)}\right)  \lambda_{min}(T)  \leq \lambda_k(ST) \leq \min_{\substack{F\subset \mathbb{R}^n \\ \dim(F)=k}} \left( \max_{x\in F\backslash \{0\}} \frac{(S\sqrt{T}x,\sqrt{T}x)}{(\sqrt{T}x,\sqrt{T}x)}\right) \lambda_{max}(T)
\end{equation*}
If $ \lbrace e_1,...,e_k \rbrace$ is a basis for $F$, we define $F'=span ( \sqrt{T}^{-1}e_1,...,\sqrt{T}^{-1}e_k )$ and we have
\begin{equation*}
    \min_{\substack{F\subset \mathbb{R}^n \\ \dim(F)=k}} \left( \max_{x\in F\backslash \{0\}} \frac{(S\sqrt{T}x,\sqrt{T}x)}{(\sqrt{T}x,\sqrt{T}x)}\right) = \min_{\substack{F'\subset \mathbb{R}^n \\ \dim(F')=k}} \left( \max_{x\in F'\backslash \{0\}} \frac{(Sx,x)}{(x,x)}\right)=\lambda_k(S)
\end{equation*}
and the desired result follows.
\end{proof}

\improvedconditionnumber*

\begin{proof}
Using Lemma \ref{le:eig commute}, we have that the eigenvalues of the $d \times d$ matrix $\nabla^2 f$ are equal to the non-zero eigenvalues of the $m \times m$ matrix $\nabla^2 \ell A A^T$. Using Corollary \ref{le:rank}, we have that the matrix $A A^T$ is of rank $d$, thus the matrix $\nabla^2 \ell A A^T$ is also of rank $d$. This means that it has exactly $m-d$ zero eigenvalues. Exactly the same holds for the matrix $A A^T$. We use also Lemma \ref{le:eig of product} for the positive definite matrix $\nabla^2 \ell$ and the positive semi-definite matrix $A A^T$ and we have:
\begin{itemize}
    \item The maximum eigenvalue of the matrix $\nabla^2 f$ is equal to the maximum eigenvalue of the matrix $\nabla^2 \ell A A^T$. For that we have
    \begin{equation*}
        \lambda_{max}(\nabla^2 \ell A A^T) \leq \lambda_{max}(\nabla^2 \ell) \lambda_{max}(A A^T).
    \end{equation*}
    Similarly the maximum eigenvalue of $A A^T$ is equal to the maximum one of $A^T A$ and we finally have
    \begin{equation*}
    \label{eq:max eig}
        \lambda_{max}(\nabla^2 f) \leq \frac{\gamma_{\ell}}{n} \lambda_{max}(A^T A)= \gamma_{\ell} \lambda_{max}\left(\frac{1}{n} A^T A \right).
    \end{equation*}
    \item The minimum eigenvalue of the matrix $\nabla^2 f$ is equal to the eigenvalue of the matrix $\nabla^2 \ell A A^T$ of order $m-d+1$. Using Lemma \ref{le:eig of product}, we have
    \begin{equation*}
        \lambda_{m-d+1}(\nabla^2 \ell A A^T) \geq \lambda_{min}(\nabla^2 \ell) \lambda_{m-d+1}(A A^T).
    \end{equation*}
    By using similar arguments as before, we have that 
    \begin{equation*}
        \lambda_{m-d+1}(A A^T)=\lambda_{min}(A^T A).
    \end{equation*}
    Thus, we finally have
    \begin{equation*}
    \label{min eig}
        \lambda_{min}(\nabla^2 f) \geq \frac{\mu_{\ell}}{n} \lambda_{min}(A^T A)=\mu_{\ell} \lambda_{min} \left( \frac{1}{n} A^T A \right).
    \end{equation*}
\end{itemize}
\end{proof}

\section{Gradient Descent with Preconditioning for GLMs}
\label{app:precond_gradient}
Gradient descent for a $\gamma$-smooth and $\mu$-strongly convex function $f(x) =\ell(Ax): \mathbb{R}^d \rightarrow \mathbb{R}$ preconditioned by a matrix $M \in \mathbb{R}^{d \times d}$ reads
\begin{align*}
    &x^{(t+1)}=x^{(t)}-\eta M^{-1} \nabla f(x^{(t)}), \\  &x^{(0)} \in \mathbb{R}^d.
\end{align*}
In our setting the matrix $M:=\frac{1}{n} A^T A$ is invertible, because we have assumed that the matrix $A$ is of full rank. The convergence is now improved up to the condition number of $M$. For the proof we follow the technique presented in \cite{chen2019gradient} for (non-preconditioned) gradient descent.
\begin{Proposition}
\label{prop:exact_precond}
The iterates $x^{(t)}$ of the previous algorithm with $\eta=\frac{2}{\mu_{\ell}+\gamma_{\ell}}$ satisfy
\begin{equation*}
    \|x^{(t)}-x^*\| \leq \left(1-\frac{1}{\kappa_{\ell}} \right)^t \|x^{(0)}-x^*\|
\end{equation*}
\end{Proposition}
\begin{proof}
Similarly to the previous argument, we have
\begin{align*}
    x^{(t+1)}-x^*&=x^{(t)}-\eta M^{-1} \nabla f(x^{(t)})-x^*=(x^{(t)}-x^*)-\eta M^{-1} \left(\int_0^1 \nabla^2 f(x(\xi)) d\xi \right)(x^{(t)}-x^*) \\ &=\left(Id-\eta \int_0^1 M^{-1} \nabla^2 f(x(\xi)) d\xi \right) (x^{(t)}-x^*)
\end{align*}
where
\begin{equation*}
    x(\xi)=x^{(t)}+\xi(x^*-x^{(t)})
\end{equation*}
Thus
\begin{equation*}
    \|x^{(t+1)}-x^*\| \leq \left \|Id-\eta \int_0^1  M^{-1} \nabla^2 f(x(\xi)) d\xi \right\| \|x^{(t)}-x^*\|\leq \max_{0 \leq \xi \leq 1} \|Id-\eta M^{-1} \nabla^2 f(x(\xi))\| \|x^{(t)}-x^*\|
\end{equation*}
Now we can write
\begin{equation*}
    M^{-1} \nabla^2 f(x(\xi))=M^{-1} A^T \nabla^2 \ell (Ax(\xi)) A
\end{equation*}
By Lemma \ref{le:eig commute}, the eigenvalues of the last matrix are exactly the same with the non-zero eigenvalues of the matrix $\nabla^2 \ell (Ax(\xi)) A M^{-1} A^T$. This matrix is $m \times m$ with rank $d$, thus it has exactly $m-d$ zero eigenvalues. The same holds for the matrix $A M^{-1} A^T$. Again by applying Lemma \ref{le:eig commute}, we know that $A M^{-1} A^T$ has $m-d$ eigenvalues equal to $0$ and the others are exactly equal with the ones of $ M^{-1} A^T A=n \textnormal{Id}$, i.e. they are all equal to $n$.
\newline
Thus, we have
\begin{equation*}
    \lambda_{max}(M^{-1} A^T \nabla^2 \ell (Ax(\xi)) A )=\lambda_{max}(\nabla^2 \ell (Ax(\xi)) A M^{-1} A^T) \leq \lambda_{max}(\nabla^2 \ell) \lambda_{max}(A M^{-1} A^T) = \frac{\gamma_{\ell}}{n} n = \gamma_{\ell}
\end{equation*}
and
\begin{equation*}
    \lambda_{min}(M^{-1} A^T \nabla^2 \ell (Ax(\xi)) A )=\lambda_{m-d+1}(\nabla^2 \ell (Ax(\xi)) A M^{-1} A^T) \geq \lambda_{min}(\nabla^2 \ell) \lambda_{m-d+1}(A M^{-1} A^T) = \frac{\mu_{\ell}}{n} n = \mu_{\ell}
\end{equation*}
by Lemma \ref{le:eig of product}, because $\nabla^2 \ell$ is positive definite and $A M^{-1} A^T$ is positive semi-definite.
\newline
Since we choose $\eta=\frac{2}{\mu_{\ell}+\gamma_{\ell}}$, the maximum eigenvalues of the matrix $\eta M^{-1} \nabla^2 f(x(\xi))$ is $\frac{2 \gamma_{\ell}}{\mu_{\ell}+\gamma_{\ell}}$ and the minimum one is $\frac{2 \mu_{\ell}}{\mu_{\ell}+\gamma_{\ell}}$.
Thus, the maximum eigenvalue of $Id-\eta M^{-1} \nabla^2 f(x(\xi))$ is less or equal than $\max \left\lbrace \frac{2 \gamma_{\ell}}{\mu_{\ell}+\gamma_{\ell}}-1, 1-\frac{2 \mu_{\ell}}{\mu_{\ell}+\gamma_{\ell}} \right  \rbrace =  \frac{\gamma_{\ell}-\mu_{\ell}}{\gamma_{\ell}+\mu_{\ell}} \leq 1-\frac{1}{\kappa_{\ell}}$. Thus 
\begin{equation*}
    \|x^{t+1}-x^*\| \leq \left(1-\frac{1}{\kappa_{\ell}}\right)
    \|x^{(t)}-x^*\|
\end{equation*}
 and by an induction argument we get the desired result.
\end{proof}

\section{Proofs of convergence for GLMs}
\label{app:convergence_GLM}
We prove the convergence result of the preconditioned algorithm for GLMs. We recall firstly the algorithm in compact form:
\begin{algorithm}[H]
\caption{Quantized Preconditioned Gradient Descent for GLM training}
\label{algo:GLM_algorithm}
\begin{algorithmic}[1]
\STATE $\bar M_i=\phi^{-1}\left(Q\left(\phi(M_i), \phi(M_{i_0}), 2 \sqrt{d} n \lambda_{max}(M) , \frac{\lambda_{min}(M)}{16 \sqrt{2} \kappa_{\ell}} \right) \right)$
\STATE $S= \frac{1}{n} \sum_{i=1}^n \bar M_i$
\STATE $\bar M=\phi^{-1} \left(Q(\phi(S),\phi(M_i), \sqrt{d} \left( \frac{\lambda_{min}(M)}{16 \kappa_{\ell}}+2 n \lambda_{max}(M) \right),\frac{\lambda_{min}(M)}{16 \sqrt{2} \kappa_{\ell}}) \right)$
\STATE $x^{(0)} \in \mathbb{R}^d, \max_i \lbrace \|x^{(0)}-x^*\|,\|x^{(0)}-x_i^*\| \rbrace \leq D$ 
\STATE $v_i^{(0)}=Q \left(\bar M^{-1} \nabla f_i(x^{(0)}),\bar M^{-1} \nabla f_{i_0}(x^{(0)}),4 n \kappa(M)  R^{(0)},\frac{\delta {R^{(0)}}}{2} \right)$
\STATE $r^{(0)}=\frac{1}{n} \sum_{i=1}^n v_i^{(0)}$.
\STATE $v^{(0)}=Q \left(r^{(0)},\bar M^{-1} \nabla f_i(x^{(0)}),\left(\frac{\delta}{2}+4 n \kappa(M) \right) R^{(0)},\frac{\delta R^{(0)}}{2} \right)$
\FOR{$t \geq 0$} 
\STATE $x^{(t+1)}=x^{(t)}-\eta v^{(t)}$
\STATE $v_i^{(t+1)}=Q \left(\bar M^{-1} \nabla f_i(x^{(t+1)}), v_i^{(t)},4 n \kappa(M) R^{(t+1)}, \frac{\delta R^{(t+1)}}{2}\right)$
\STATE $r^{(t+1)}=\frac{1}{n} \sum_{i=1}^n v_i^{(t+1)}$
\STATE $v^{(t+1)}=Q \left(r^{(t+1)},v^{(t)},(\frac{\delta}{2}+4 n \kappa(M)) R^{(t+1)}, \frac{\delta R^{(t+1)}}{2} \right)$
\ENDFOR
\end{algorithmic}

\end{algorithm}
\begin{lemma}
\label{le:inexact_precond}
Consider the algorithm
\begin{equation*}
    x^{(t+1)}=x^{(t)}-\eta \bar M^{-1} \nabla f(x^{(t)})
\end{equation*}
starting from a point $x^{(0)} \in \mathbb{R}^d$ such that $\|x^{(0)}-x^*\| \leq D$, where $\eta=\frac{2}{\mu_{\ell}+\gamma_{\ell}}$ and $\bar M$ is the quantized estimation of $M$ obtained in Algorithm \ref{algo:GLM_algorithm}. Then, the iterates of this algorithm satisfy
\begin{equation*}
    \|x^{(t)}-x^*\| \leq \left(1-\frac{1}{2 \kappa_{\ell}} \right)^t D.
\end{equation*}
\end{lemma}

\begin{proof}
We use the same proof technique as in Proposition \ref{prop:exact_precond}, with the difference that now we have the quantized estimation $\bar M$ of $M$ instead of the original:
\begin{align*}
    x^{(t+1)}-x^*&=x^{(t)}-\eta \bar M^{-1} \nabla f(x^{(t)})-x^*=(x^{(t)}-x^*)-\eta \bar M^{-1} \left(\int_0^1 \nabla^2 f(x(\xi)) d\xi \right)(x^{(t)}-x^*) \\ &=\left(Id-\eta \int_0^1 \bar M^{-1} \nabla^2 f(x(\xi)) d\xi \right) (x^{(t)}-x^*)
\end{align*}
where
\begin{equation*}
    x(\xi)=x^{(t)}+\xi(x^*-x^{(t)}).
\end{equation*}
Thus
\begin{equation*}
    \|x^{(t+1)}-x^*\| \leq \left \|Id-\eta \int_0^1  \bar M^{-1} \nabla^2 f(x(\xi)) d\xi \right\| \|x^{(t)}-x^*\|\leq \max_{0\leq \xi \leq 1} \|Id-\eta \bar M^{-1} \nabla^2 f(x(\xi))\| \|x^{(t)}-x^*\|.
\end{equation*}
Now we can write
\begin{equation*}
    \bar M^{-1} \nabla^2 f(x(\xi))=M^{-1} \nabla^2 f(x(\xi))+(\bar M^{-1}- M^{-1}) \nabla^2 f(x(\xi))
\end{equation*}
and
\begin{equation*}
    \|Id-\eta \bar M^{-1} \nabla^2 f(x(\xi))\| \leq \|Id-\eta M^{-1} \nabla^2 f(x(\xi))\|+\eta \| (\bar M^{-1}- M^{-1}) \nabla^2 f(x(\xi))\|.
\end{equation*}
For the matrix $M^{-1} A^T \nabla \ell^2 (Ax(\xi)) A$ we apply exactly the same argument as in Proposition \ref{prop:exact_precond} and have
\begin{equation*}
\max_{0\leq \xi \leq 1} \|Id-\eta M^{-1} \nabla^2 f(x(\xi))\| \leq \frac{\gamma_{\ell}-\mu_{\ell}}{\gamma_{\ell}+\mu_{\ell}}<1-\frac{1}{\kappa_{\ell}}.
\end{equation*}
For the extra error term, we firstly have to study the quantization error $\|M-\bar M\|$:
\newline
Notice that
\begin{equation*}
   \| \phi(M_i)-\phi(M_{i_0}) \| \leq \sqrt{d} \|M_i-M_{i_0}\| \leq \sqrt{d}(\| M_i \|+\| M_{i_0} \|) \leq 2 \sqrt{d} n \lambda_{max}(M)
\end{equation*}
which implies that
\begin{equation*}
    \| \phi(\bar M_i)-\phi(M_i) \| \leq \frac{\lambda_{min}(M)}{16 \sqrt{2} \kappa_{\ell}}
\end{equation*}
by the definition of quantization parameters (we have $\lambda_{max}(M_i) \leq n \lambda_{max}(M)$, because $n M=\sum_{i=1}^n M_i$ and every $M_i$ is positive semi-definite). The last inequality implies
\begin{equation*}
    \| \bar M_i-M_i \| \leq \frac{\lambda_{min}(M)}{16 \kappa_{\ell}}
\end{equation*}
and
\begin{equation*}
    \|S-M\| \leq \frac{1}{n} \sum_{i=1}^n \| \bar M_i-M_i \| \leq \frac{\lambda_{min}(M)}{16 \kappa_{\ell}}.
\end{equation*}
Now we can write
\begin{equation*}
    \|\phi(S)-\phi(M_i)\| \leq \sqrt{d} \|S-M_i\| \leq  \sqrt{d} (\|S-M\|+\|M-M_i\|) \leq \sqrt{d} \left( \frac{\lambda_{min}(M)}{16 \kappa_{\ell}}+2 n \lambda_{max}(M) \right) \leq 3 n \sqrt{d} \lambda_{max}(M).
\end{equation*}
By the definition of quantization parameters, this implies
\begin{equation*}
    \|\phi(\bar M)-\phi(S)\| \leq \frac{\lambda_{min}(M)}{16 \sqrt{2} \kappa_{\ell}}
\end{equation*}
and concequently
\begin{equation*}
    \|\bar M-S\| \leq \frac{\lambda_{min}(M)}{16 \kappa_{\ell}}.
\end{equation*}
Then
\begin{equation*}
    \|M-\bar M\| \leq \|M-S\|+\|S-\bar M\| \leq \frac{\lambda_{min}(M)}{8 \kappa_{\ell}}.
\end{equation*}
By standard results in perturbation theory, we know that
\begin{equation*}
    \lambda_{min}(\bar M) \geq \lambda_{min}(M)-\|M-\bar M\| \geq \lambda_{min}(M)-\frac{\lambda_{min}(M)}{8 \kappa_{\ell}} \geq \frac{\lambda_{min}(M)}{2}.
\end{equation*}
This implies
\begin{equation*}
    \| \bar M^{-1} \|=\lambda_{max}(\bar M^{-1})=\frac{1}{\lambda_{min}(\bar M^{-1})} \leq \frac{2}{\lambda_{min}(M)}.
\end{equation*}
Now we have
\begin{align*}
    &\max_{0\leq \xi \leq 1} \frac{2}{\mu_{\ell}+\gamma_{\ell}} \|(\bar M^{-1}- M^{-1}) \nabla^2 f(x(\xi))\| \leq \max_{0\leq \xi \leq 1} \frac{2}{\mu_{\ell}+\gamma_{\ell}} \| \bar M^{-1} (M-\bar M) M^{-1} \nabla^2 f(x(\xi))\| \\ & \leq \max_{0\leq \xi \leq 1} \frac{2}{\mu_{\ell}+\gamma_{\ell}} \| \bar M^{-1} (M-\bar M)\| \|M^{-1} \nabla^2 f(x(\xi))\| \leq \frac{2}{\mu_{\ell}+\gamma_{\ell}} \| \bar M^{-1} \| \| M-\bar M \| \gamma_{\ell} \\ & \leq \frac{4}{\lambda_{min}(M)} \frac{\lambda_{min}(M)}{8\kappa_{\ell}} =\frac{1}{2 \kappa_{\ell}}.
\end{align*}
Thus, it holds
\begin{equation*}
    \|x^{(t+1)}-x^*\| \leq \left(1-\frac{1}{2 \kappa_{\ell}} \right) \|x^{(t)}-x^*\|
\end{equation*}
which implies
\begin{equation*}
    \|x^{(t+1)}-x^*\| \leq \left(1-\frac{1}{2 \kappa_{\ell}} \right)^t \|x^{(0)}-x^*\| \leq \left(1-\frac{1}{2 \kappa_{\ell}} \right)^t D.
\end{equation*}
\end{proof}
We recall the parameters
\begin{align*}
   & \xi=1-\frac{1}{2 \kappa_{\ell}}, \\ 
   & K=\frac{2}{\xi}, \\
   & \delta=\frac{\xi(1-\xi)}{4}, \\
   & R^{(t)}= \frac{\gamma_{\ell}}{2} K \left(1-\frac{1}{4 \kappa_{\ell}} \right)^t D.
\end{align*}
\begin{lemma}
The iterates of Algorithm \ref{algo:GLM_algorithm} satisfy the following inequalities:
\begin{align*}
    &\| x^{(t)}-x^* \| \leq \left(1-\frac{1}{4\kappa_{\ell}}\right)^t D, \\ & \| \bar M^{-1} \nabla f_i(x^{(t)})- v^{(t)}_i \| \leq \frac{\delta R^{(t)}}{2},
    \\& \|\bar M^{-1} \nabla f(x^{(t)})-v^{(t)} \| \leq \delta R^{(t)}.
\end{align*}
\end{lemma}
\begin{proof}
We firstly prove the inequalities for $t=0$. The first one is direct by the definition of $D$. For the second one, we notice that
\begin{align*}
    &\| \bar M^{-1} \nabla f_i(x^{(0)})-\bar M^{-1} \nabla f_{i_0}(x^{(0)}) \| \leq \frac{2}{\lambda_{min}(M)} (\|\nabla f_i(x^{(0)})\|+\|\nabla f_{i_0}(x^{(0)})\|) \leq \\ & \frac{2}{\lambda_{min}(M)} (\gamma_i \|x^{(0)}-x_i^*\|+\gamma_{i_0} \|x^{(0)}-x_{i_0}^*\|) \leq 2 \gamma_{\ell} \frac{\lambda_{max}(M_i)}{\lambda_{min}(M)} D + 2 \gamma_{\ell} \frac{\lambda_{max}(M_{i_0})}{\lambda_{min}(M)} D \leq 4 n \kappa(M) R^{(0)}.
\end{align*}
The last inequality follows because $K \geq 2$ and $\lambda_{max}(M_i) \leq n \lambda_{max}(M)$.
\newline
(We recall also that $\| \bar M^{-1} \| \leq \frac{2}{\lambda_{min}(M)}$, because $\lambda_{min}(\bar M) \geq \lambda_{min}(M)-\|M-\bar M\| \geq \lambda_{min}(M)/2$.)
\newline
By the definition of $v_i^{(0)}$, we have
\begin{equation*}
    \|v_i^{(0)}-\bar M^{-1} \nabla f_i(x^{(0)})\| \leq \frac{\delta R^{(0)}}{2}.
\end{equation*}
Towards the third inequality at $t=0$, we have
\begin{equation*}
    \| r^{(0)}-\bar M^{-1} \nabla f(x^{(0)}) \| \leq \frac{1}{n} \sum_{i=1}^n \|v_i^{(0)}-\bar M^{-1} \nabla f_i(x^{(0)})\| \leq \frac{\delta R^{(0)}}{2}.
\end{equation*}
Also, it holds
\begin{equation*}
    \|r^{(0)}-\bar M^{-1} \nabla f_i(x^{(0)})\| \leq \|r^{(0)}-\bar M^{-1} \nabla f(x^{(0)})\|+\|\bar M^{-1} \nabla f(x^{(0)})-\bar M^{-1} \nabla f_i(x^{(0)})\| \leq  \left(\frac{\delta}{2} + 4 n \kappa(M) \right) R^{(0)} ,
\end{equation*}
thus, by the definition of $v^{(0)}$,
\begin{equation*}
    \|v^{(0)}-r^{(0)}\| \leq \frac{\delta R^{(0)}}{2}
\end{equation*}
and putting everything together, we have
\begin{equation*}
    \|v^{(0)}-\bar M^{-1} \nabla f(x^{(0)})\| \leq \|v^{(0)}-r^{(0)}\|+\|r^{(0)}-\bar M^{-1} \nabla f(x^{(0)})\| \leq \frac{\delta R^{(0)}}{2}+\frac{\delta R^{(0)}}{2}=\delta R^{(0)}.
\end{equation*}

Now we assume that the inequalities hold for $t$ and prove that they continue to hold for $t+1$. We start with the first one:
\begin{align*}
    \|x^{(t+1)}-x^*\|&=\|x^{(t)}-\eta v^{(t)}+\eta \bar M^{-1} \nabla f(x^{(t)})-\eta \bar M^{-1} \nabla f(x^{(t)})-x^*\| \\& \leq \eta \|\bar M^{-1} \nabla f(x^{(t)})-v^{(t)}\|+\|x^{(t)}-\eta \bar M^{-1} \nabla f(x^{(t)})-x^*\| \\ & \leq \frac{2}{\gamma_{\ell}} \delta R^{(t)}+\xi \left(1-\frac{1}{4 \kappa_{\ell}} \right)^t D \\ &= \frac{2}{\gamma_{\ell}} \delta \frac{\gamma_{\ell}}{2} K\left(1-\frac{1}{4 \kappa_{\ell}} \right)^t D+\xi \left(1-\frac{1}{4 \kappa_{\ell}} \right)^t D \\ &=\delta K \left(1-\frac{1}{4 \kappa_{\ell}} \right)^t D + \xi \left(1-\frac{1}{4 \kappa_{\ell}} \right)^t D \\ &= (\delta K + \xi) \left(1-\frac{1}{4 \kappa_{\ell}} \right)^t D= \left(1-\frac{1}{4 \kappa_{\ell}} \right)^{t+1} D.
\end{align*}
For the second inequality it suffices to show that
\begin{equation*}
    \| \bar M^{-1} \nabla f_i(x^{(t+1)})- v^{(t)}_i \| \leq 4 n \kappa(M) R^{(t+1)}.
\end{equation*}
To that end, we write
\begin{align*}
    \| \bar M^{-1} \nabla f_i(x^{(t+1)})- v^{(t)}_i \| & = \| \bar M^{-1} \nabla f_i(x^{(t+1)})- \bar M^{-1}\nabla f_i(x^{(t)})+ \bar M^{-1}\nabla f_i(x^{(t)})- v^{(t)}_i \| \\ & \leq \| \bar M^{-1} \nabla f_i(x^{(t+1)})- \bar M^{-1}\nabla f_i(x^{(t)}) \|+ \|\bar M^{-1}\nabla f_i(x^{(t)})-v^{(t)}_i \| \\ & \leq \gamma_i \| \bar M^{-1} \| \| x^{(t+1)}-x^{(t)} \|+ \delta R^{(t)} \\ & \leq
    \gamma_{\ell} \lambda_{max}(M_i) \frac{2}{\lambda_{min}(M)} (\|x^{(t+1)}-x^*\|+\|x^{(t)}-x^*\|)+ \delta R^{(t)} \\ & \leq 4 n \gamma_{\ell} \kappa(M) \left(1-\frac{1}{4 \kappa_{\ell}} \right)^t D + \delta \frac{\gamma_{\ell}}{2} K \left(1-\frac{1}{4 \kappa_{\ell}} \right)^t D \\ & \leq   2 n (2/K+\delta/4)K \gamma_{\ell} \kappa(M) \left(1-\frac{1}{4 \kappa_{\ell}} \right)^t D \\ & \leq 2 n (2/K+\delta K)K \gamma_{\ell} \kappa(M) \left(1-\frac{1}{4 \kappa_{\ell}} \right)^t D   \\ & \leq 4 n \kappa(M) R^{(t+1)}.
\end{align*}
Previously we have used again that $\lambda_{max}(M_i) \leq n \lambda_{max}(M)$, because $n M=\sum_{i=1}^n M_i$ and all matrices $M_i$ are positive semi-definite.
\newline
For the last inequality we have
\begin{align*}
\| \bar M^{-1} \nabla f(x^{(t+1)})-r^{(t+1)} \| \leq \frac{1}{n} \sum_{i=1}^n \| \bar M^{-1} \nabla f_i(x^{(t+1)})-v^{(t+1)}_i \| \leq \frac{\delta R^{(t+1)}}{2}   
\end{align*}
and
\begin{align*}
    \| r^{(t+1)}-v^{(t)} \| &= \| r^{(t+1)}- \bar M^{-1} \nabla f(x^{(t+1)})+\bar M^{-1}\nabla f(x^{(t+1)})-\bar M^{-1}\nabla f(x^{(t)})+\bar M^{-1}\nabla f(x^{(t)})-v^{(t)} \| \\ & \leq \| r^{(t+1)}-\bar M^{-1}\nabla f(x^{(t+1)}) \|+ \| \bar M^{-1}\nabla f(x^{(t+1)})-\bar M^{-1}\nabla f(x^{(t)}) \| + \| \bar M^{-1} \nabla f(x^{(t)})-v^{(t)} \| \\ & \leq
    \frac{\delta R^{(t+1)}}{2}+\gamma \frac{2}{\lambda_{min}(M)} \| x^{(t+1)}-x^{(t)} \|+ \delta R^{(t)} \\ & \leq \frac{\delta R^{(t+1)}}{2}+4 \gamma_{\ell} \kappa(M) (\| x^{(t+1)}-x^*\|+\| x^{(t)}-x^* \|)+ \delta R^{(t)} \\ & \leq
    \frac{\delta R^{(t+1)}}{2}+4 \kappa(M) R^{(t+1)} \leq \left(4 n \kappa(M) + \frac{\delta}{2} \right) R^{(t+1)}.  
\end{align*}
The last part of the inequality follows from the same argument used in deriving the second one.
\newline
The last inequality implies that
\begin{equation*}
    \| v^{(t+1)}-r^{(t+1)} \| \leq \frac{\delta R^{(t+1)}}{2}.
\end{equation*}
Thus, putting everything together, we have
\begin{equation*}
    \| \bar M^{-1} \nabla f(x^{(t+1)})-v^{(t+1)} \| \leq \|\bar M^{-1} \nabla f(x^{(t+1)})-r^{(t+1)} \|+\| r^{(t+1)}-v^{(t+1)}\| \leq \frac{\delta R^{(t+1)}}{2}+\frac{\delta R^{(t+1)}}{2} = \delta R^{(t+1)}.
\end{equation*}
\end{proof}

\convergenceGLM*

\begin{proof}
The inequality for the convergence rate of the distance of the iterates from the minimizer holds from the previous lemma. We now turn our interest to the total communication cost.
We start from the quantization of the matrix $M$:
\newline
The communication cost for encoding each $M_i$ and decoding in the master node is
\begin{equation*}
 \bigO \left( \frac{d (d+1)}{2} \textnormal{log}_2 \left(\frac{2 \sqrt{d} n \lambda_{max}(M)}{\frac{\lambda_{min}(M)}{16 \sqrt{2} \kappa_{\ell}}} \right) \right)=\bigO(d^2 \textnormal{log}_2 (\sqrt{d} n \kappa_{\ell} \kappa(M))).
\end{equation*}
The communication cost of encoding $S$ in the master node and then decode back in every machine is
\begin{equation*}
 \bigO \left( \frac{d (d+1)}{2} \textnormal{log}_2 \left(\frac{3 \sqrt{d} n \lambda_{max}(M)}{\frac{\lambda_{min}(M)}{16 \sqrt{2}  \kappa_{\ell}}} \right) \right)=\bigO(d^2 \textnormal{log}_2 (\sqrt{d} n \kappa_{\ell} \kappa(M))).
\end{equation*}
Since we have $n$-many communications of each kind, the total communication cost is
\begin{equation*}
    b_m=\bigO(n d^2 \textnormal{log}_2 (\sqrt{d}n \kappa_{\ell} \kappa(M))=\bigO(n d^2 \textnormal{log} (\sqrt{d}n \kappa_{\ell} \kappa(M)). 
\end{equation*}
The communication cost of quantizing the descent direction $v^{(t)}$ at step $t \geq 0$ is at most 
\begin{equation*}
    \bigO \left(n d \log_2 \frac{4 n \kappa(M)}{\delta/2} \right)= \bigO \left(n d \log \frac{n \kappa(M)}{\delta} \right)
\end{equation*}
for encoding the local descent directions and 
\begin{equation*}
    \bigO \left(n d \log_2 \frac{4 n \kappa(M)+\frac{\delta}{2}}{\delta/2} \right) \leq \bigO \left(n d \log \frac{9 n \kappa(M)}{\delta} \right)= \bigO \left(n d \log \frac{ n \kappa(M)}{\delta} \right) 
\end{equation*}
for decoding back. 
Since we have
\begin{equation*}
    \frac{1}{\delta} = \frac{4}{\xi(1-\xi)} = \frac{4}{\frac{1}{2 \kappa_{\ell}} \left(1-{\frac{1}{2 \kappa_{\ell}}} \right)} \leq 16 \kappa_{\ell}, 
\end{equation*}
we can bound the total communication cost by
\begin{equation*}
    b=\bigO \left(n d \log (n \kappa_{\ell} \kappa(M) \right).
\end{equation*}
We have $f(x^{(t)})-f(x^*) \leq \epsilon$ if $\|x^{(t)}-x^*\| \leq \sqrt{\frac{2 \epsilon}{\gamma}}$, thus we reach accuracy $\epsilon$ in terms of function values in at most $t=2 \kappa_{\ell} \log \frac{\gamma D^2}{2 \epsilon}$ and putting everything together we find the total communication cost for quantizing the descent directions along the whole optimization process to be
\begin{equation*}
    \bigO \left(n d \kappa_{\ell} \log (n \kappa_{\ell} \kappa(M)) \log \frac{\gamma D^2}{2 \epsilon}\right).
\end{equation*}
Thus, the total communication cost in number of bits is obtained by summing the cost for matrix and descent direction quantization:
\begin{equation*}
    b=   \bigO \left(n d^2 \log \left(\sqrt{d} n \kappa_{\ell} \kappa(M) \right) \right) +  \bigO \left(n d \kappa_{\ell} \log (n \kappa_{\ell} \kappa(M)) \log \frac{\gamma D^2}{\epsilon}\right).
\end{equation*}
\end{proof}

\section{Proofs for Quantized Newton's Method}
\label{app:quant_newton}
We firstly recall Quantized Newton's method in a compact form as we did also for GLMs:

\begin{algorithm}[H]
\caption{Quantized Newton's Method}
\label{algo:Newton}
\begin{algorithmic}[1]
\STATE $x^{(0)} \in \mathbb{R}^d, \max_i \lbrace \|x^{(0)}-x^*\|,\|x^{(0)}-x_i^*\| \rbrace \leq \frac{\alpha \mu}{2 \sigma}$
\STATE $H_0^i= \phi^{-1}\left(Q \left(\phi(\nabla^2 f_i(x^{(0)})), \phi(\nabla^2 f_{i_0}(x^{(0)})), 2 \sqrt{d} \gamma, \frac{G^{(0)}}{2 \sqrt{2}\kappa} \right)\right)$
\STATE $S_0=\frac{1}{n} \sum_{i=1}^n H_0^i$
\STATE $H_0= \phi^{-1}\left(Q \left(\phi(S_0), \phi(\nabla^2 f_i(x^{(0)})),  \sqrt{d} \left( \frac{G^{(0)}}{2 \kappa} + 2 \gamma \right) , \frac{G^{(0)}}{2 \sqrt{2} \kappa} \right)\right)$
\STATE $v_i^{(0)}=Q \left(H_0^{-1} \nabla f_i(x^{(0)}), H_0^{-1} \nabla f_{i_0} (x^{(0)}),4 \kappa P^{(0)},\frac{\theta P^{(0)}}{2} \right)$
\STATE $p^{(0)}=\frac{1}{n} \sum_{i=1}^n v_i^{(0)}$
\STATE $v^{(0)}=Q \left(P^{(0)}, H_0^{-1} \nabla f_i (x^{(0)}),\left(\frac{\theta}{2}+4 \kappa \right) P^{(0)}, \frac{\theta P^{(0)}}{2} \right)$
\FOR{$t \geq 0$} 
\STATE $x^{(t+1)}=x^{(t)}-v^{(t)}$
\STATE $H_{t+1}^i=\phi^{-1}\left(Q \left(\phi(\nabla^2 f_i(x^{(t+1)})), \phi(H_t^i), \frac{10 \sqrt{d}}{1+\alpha} G^{(t+1)}, \frac{G^{(t+1)}}{2 \sqrt{2} \kappa} \right)\right)$
\STATE $S_{t+1}=\frac{1}{n} \sum_{i=1}^{n} H_{t+1}^i$
\STATE $H_{t+1}=\phi^{-1} \left(Q \left(\phi(S_{t+1}), \phi(H_t) , \sqrt{d} \left( \frac{1}{2 \kappa}+ \frac{10}{1+\alpha} \right) G^{(t+1)} , \frac{G^{(t+1)}}{2 \sqrt{2} \kappa} \right) \right)$
\STATE $v^{(t+1)}_i=Q \left(H_{t+1}^{-1} \nabla f_i(x^{(t+1)}), v^{(t)}_i, 11 \kappa P^{(t+1)} , \frac{\theta P^{(t+1)}}{2} \right)$
\STATE $p^{(t+1)}=\frac{1}{n} \sum_{i=1}^n v_i^{(t+1)}$
\STATE $v^{(t+1)}=Q \left(r^{(t+1)}, v^{(t)},\left(\frac{\theta}{2}+11 \kappa \right) P^{(t+1)}, \frac{\theta P^{(t+1)}}{2} \right)$
\ENDFOR
\end{algorithmic}

\end{algorithm}

We recall the parameters
\begin{align*}
    &G^{(t)}=\frac{\mu}{4} \alpha \left(\frac{1+\alpha}{2} \right)^t, \\
    &\alpha \geq 2 \frac{\sigma}{\mu} \max_i \lbrace \|x^{(0)}-x^*\|,\|x^{(0)}-x_i^*\| \rbrace \\
    &\theta=\frac{\alpha(1-\alpha)}{4}, \\
    & K=\frac{2}{\alpha}, \\
    & P^{(t)}= \frac{\mu}{2 \sigma} K \alpha \left( \frac{1+\alpha}{2}\right)^t.
\end{align*}

\begin{restatable}{lemma}{newtondescquant}
\label{le:desc_direction_newton}
    The iterates $x^{(t)}$ of the quantized Newton's algorithm satisfy the inequalities
\begin{align*}
    & \|x^{(t)}-x^*\| \leq \frac{\mu}{2 \sigma} \alpha \left(\frac{1+\alpha}{2} \right)^t, \\
    & \|H_t^i-\nabla^2 f_i(x^{(t)})\| \leq \frac{G^{(t)}}{2 \kappa}, \\
    & \|H_t- \nabla^2 f(x^{(t)})\| \leq \frac{G^{(t)}}{\kappa}, \\
    & \|H_t^{-1} \nabla f_i(x^{(t)})-v^{(t)}_i\| \leq \frac{\theta P^{(t)}}{2}, \\
    & \|H_t^{-1} \nabla f(x^{(t)})-v^{(t)}\| \leq \theta P^{(t)}.
\end{align*}
\end{restatable}

\begin{proof}
We firstly prove that the inequalities hold at $t=0$. The first one is trivial by the choice of $x^{(0)}$.

For the second one, it suffices to show that
\begin{equation*}
    \| \phi(H_0^i)-\phi(\nabla^2 f_i(x^{(0)})) \| \leq \frac{G^{(0)}}{2\sqrt{2} \kappa}
\end{equation*}
by Lemma \ref{le:norm_distortion},
and for that suffices
\begin{equation*}
    \| \phi(\nabla^2 f_i(x^{(0)}))-\phi(\nabla^2 f_{i_0}(x^{(0)})) \| \leq 2 \sqrt{d} \gamma,
\end{equation*}
which is indeed the case because
\begin{align*}
    &\| \phi(\nabla^2 f_i(x^{(0)}))-\phi(\nabla^2 f_{i_0}(x^{(0)})) \| \leq \sqrt{d} \|\nabla^2 f_i(x^{(0)})-\nabla^2 f_{i_0}(x^{(0)})\| \leq \sqrt{d} (\|\nabla^2 f_i(x^{(0)})\|+\|\nabla^2 f_{i_0}(x^{(0)})\|)  \leq 2 \sqrt{d} \gamma,
\end{align*}
again using Lemma \ref{le:norm_distortion}.

For the third inequality at $t=0$, we have
\begin{equation*}
    \| \nabla^2 f(x^{(0)}) - S_0 \| \leq \frac{1}{n} \sum_{i=1}^n \|\nabla^2 f_i(x^{(0)}) - H_0^i \| \leq \frac{G^{(0)}}{2 \kappa}.
\end{equation*}
We need also $\| S_0 - H_0 \| \leq \frac{G^{(0)}}{2 \kappa}$ and for that it suffices $\| \phi(S_0) - \phi(H_0) \| \leq \frac{G^{(0)}}{2 \sqrt{2} \kappa}$, which follows from 
\begin{equation*}
    \| \phi(S_0) - \phi(\nabla^2 f_i(x^{(0)})) \| \leq \sqrt{d} \left(\frac{G^{(0)}}{2 \kappa} + 2 \gamma \right).
\end{equation*}
In order to show the latter, we write
\begin{align*}
    &\| \phi(S_0) - \phi(\nabla^2 f_i(x^{(0)})) \| \leq \sqrt{d} \| S_0-\nabla^2 f_i(x^{(0)}) \| \leq \sqrt{d} (\| S_0-\nabla^2 f(x^{(0)}) \|+\| \nabla^2 f(x^{(0)})-\nabla^2 f_i(x^{(0)}) \|) \\ &  \leq \sqrt{d} \left(\frac{G^{(0)}}{2 \kappa} + 2 \gamma \right).
\end{align*}

For the fourth one it suffices to show that
\begin{equation*}
    \|H_0^{-1} \nabla f_i(x^{(0)})-H_0^{-1} \nabla f_{i_0}(x^{(0)})\| \leq 4 \kappa P^{(0)}.
\end{equation*}
Indeed,
\begin{align*}
    &\|H_0^{-1} \nabla f_i(x^{(0)})-H_0^{-1} \nabla f_{i_0}(x^{(0)})\| \leq \|H_0^{-1}\| (\|\nabla f_i(x^{(0)})\|+\|\nabla f_{i_0}(x^{(0)})\|) \leq \frac{2}{\mu} (\gamma \|x^{(0)}-x_i^*\|+ \gamma \|x^{(0)}-x_{i_0}^*\|) \\ & \leq 4 \frac{\gamma}{\mu} K \frac{\mu}{2 \sigma} \alpha= 4 \kappa P^{(0)}.
\end{align*}
In the previous inequality we used that $\|H_0^{-1}\| \leq \frac{2}{\mu}$ and this can be seen as follows:
\begin{align*}
   \|H_0^{-1}\|=\frac{1}{\lambda_{min}(H_0)} \leq \frac{1}{\lambda_{min}(\nabla^2 f(x^{(0)}))-\frac{G^{(0)}}{\kappa}} \leq \frac{1}{\mu-\frac{\mu}{4} \alpha} \leq \frac{2}{\mu}. 
\end{align*}

For the fifth inequality at $t=0$, we have
\begin{equation*}
    \|H_0^{-1} \nabla f(x^{(0)})-p^{(0)}\| \leq \frac{1}{n} \sum_{i=1}^n \|H_0^{-1} \nabla f_i(x^{(0)})-v_i^{(0)}\| \leq \frac{\theta P^{(0)}}{2}.
\end{equation*}
We need also
\begin{equation*}
    \| v^{(0)}-p^{(0)} \| \leq \frac{\theta P^{(0)}}{2}.
\end{equation*}
For that it suffices to show that
\begin{equation*}
  \|p^{(0)}-H_0^{-1} \nabla f_i(x^{(0)})\| \leq \left(\frac{\theta}{2}+4 \kappa \right) P^{(0)}.
\end{equation*}
Indeed
\begin{align*}
    &\|p^{(0)}-H_0^{-1} \nabla f_i(x^{(0)})\| \leq \|p^{(0)}-H_0^{-1} \nabla f(x^{(0)})\|+\|H_0^{-1} \nabla f(x^{(0)})-H_0^{-1} \nabla f_i(x^{(0)})\| \\ & \leq \frac{\theta P^{(0)}}{2}+\frac{2}{\mu} (\gamma \|x^{(0)}-x^*\|+ \gamma \|x^{(0)}-x_i^*\|) \leq \frac{\theta P^{(0)}}{2}+ 4 \kappa P^{(0)}=\left(\frac{\theta}{2}+4 \kappa \right) P^{(0)}.
\end{align*}

Now we assume that the inequalities hold for $t$ and wish to prove that they also hold for $t+1$. We start with an auxiliary result regarding taking a Newton iterate using the quantized version of the Hessian but the exact gradient: 
\begin{equation*}
    \|x^{(t)}-H_t^{-1} \nabla f(x^{(t)})-x^*\| \leq \alpha \|x^{(t)}-x^*\|.
\end{equation*}
For proving that we start by writing
\begin{align*}
    &x^{(t)}-H_t^{-1} \nabla f(x^{(t)})-x^*=(x^{(t)}-x^*)-H_t^{-1} \left(\int_0^1 \nabla^2 f(x(\xi)) d\xi \right)(x^{(t)}-x^*) = \\ &\left(Id- \int_0^1 H_t^{-1} \nabla^2 f(x(\xi)) d\xi \right) (x^{(t)}-x^*),
\end{align*}
where
\begin{equation*}
    x(\xi)=x^{(t)}+\xi(x^*-x^{(t)}).
\end{equation*}

Thus
\begin{equation*}
    \|x^{(t)}-H_t^{-1} \nabla f(x^{(t)})-x^*\| \leq \left \|Id-\int_0^1  H_t^{-1} \nabla^2 f(x(\xi)) d\xi \right\| \|x^{(t)}-x^*\|\leq \max_{0 \leq \xi \leq 1} \|Id-H_t^{-1} \nabla^2 f(x(\xi))\| \|x^{(t)}-x^*\|.
\end{equation*}

Now we need to deal with the quantity $\max_{0 \leq \xi \leq 1} \|Id-H_t^{-1} \nabla^2 f(x(\xi))\|$. 

We first write bound $\| H_t^{-1} \|$:
\begin{equation*}
    \|H_t^{-1}\|=\frac{1}{\lambda_{min}(H_t)} \leq \frac{1}{\lambda_{min}(\nabla^2 f(x_t)-\|\nabla^2 f(x_t))-H_t\|} \leq \frac{1}{\mu-\frac{G^{(t)}}{\kappa}}.
\end{equation*}
\newline
Now, we have
\begin{equation*}
    \frac{G^{(t)}}{\kappa} \leq \frac{\mu}{2}
\end{equation*}
and the result follows. This happens if
\begin{equation*}
    \frac{1}{\kappa} \leq \frac{2}{\alpha}
\end{equation*}
which holds always true, because $\frac{1}{\kappa},\alpha<1$.
Thus
\begin{equation*}
    \|H_t^{-1}\| \leq \frac{2}{\mu}.
\end{equation*}

Second, we bound the quantity $\|\nabla^2 f(x^{(t)})^{-1}-H_t^{-1}\|$:
\begin{equation*}
    \|\nabla^2 f(x^{(t)})^{-1}-H_t^{-1}\|= \|\nabla^2 f(x^{(t)})^{-1}(\nabla^2 f(x^{(t)})-H_t)H_t^{-1}\| \leq \|\nabla^2 f(x^{(t)})-H_t\| \|\nabla^2 f(x^{(t)}) ^{-1}\| \|H_t^{-1}\| \leq \frac{G^{(t)}}{\kappa} \frac{1}{\mu} \frac{2}{\mu}=\frac{2}{\mu^2} \frac{G^{(t)}}{\kappa}.
\end{equation*}

Using that and the fact that $f$ is $\mu$-strongly convex, $\gamma$-smooth, with a $\sigma$-Lipschitz Hessian, we get

\begin{align*}
   &\max_{0 \leq \xi \leq 1} \|Id-H_t^{-1} \nabla^2 f(x(\xi))\| = \max_{0 \leq \xi \leq 1} \|Id-\nabla^2 f(x^{(t)}) ^{-1} \nabla^2 f(x(\xi))+(\nabla^2 f(x^{(t)}) ^{-1} -H_t^{-1}) \nabla^2 f(x(\xi)) \| \leq \\ & \max_{0 \leq \xi \leq 1} \|Id-\nabla^2 f(x^{(t)}) ^{-1} \nabla^2 f(x(\xi)) \|+\max_{0 \leq \xi \leq 1} \|(\nabla^2 f(x^{(t)}) ^{-1} -H_t^{-1}) \nabla^2 f(x(\xi)) \| \leq \\ &  \max_{0 \leq \xi \leq 1} \|\nabla^2 f(x^{(t)})^{-1} (\nabla^2 f(x^{(t)})- \nabla^2 f(x(\xi))) \|+\|\nabla^2 f(x^{(t)}) ^{-1} -H_t^{-1}\|\max_{0 \leq \xi \leq 1} \| \nabla^2 f(x(\xi)) \| \leq \\ &
   \frac{\sigma}{\mu} \| x^{(t)}-x^* \| + \frac{2}{\mu^2} \frac{G^{(t)}}{\kappa} \gamma.
\end{align*}

Thus, we finally get
\begin{align*}
    \|x^{(t)}-H_t^{-1} \nabla f(x^{(t)})-x^*\| & \leq \frac{\sigma}{\mu} \|x^{(t)}-x^*\|^2+ G^{(t)} \frac{2}{\mu} \|x^{(t)}-x^*\| \\& \leq \frac{\sigma}{\mu} \|x^{(t)}-x^*\| \frac{\mu}{2 \sigma} \alpha \left(\frac{1+\alpha}{2} \right)^t + \frac{\mu}{4} \alpha \left(\frac{1+\alpha}{2} \right)^t \frac{2}{\mu} \|x^{(t)}-x^*\| \\ & \leq \alpha \left(\frac{1+\alpha}{2} \right)^t \|x^{(t)}-x^*\| \leq \alpha \|x^{(t)}-x^*\|,
\end{align*}
which is the desired result.

Now we pass to the exact iterate of our algorithm. Using the induction hypothesis and the previous inequality, we have

\begin{align*}
    \|x^{(t+1)}-x^*\| &=\|x^{(t)}-v^{(t)}+H_t^{-1} \nabla f(x^{(t)})-H_t^{-1} \nabla f(x^{(t)})-x^*\| \\ & \leq
    \|H_t^{-1} \nabla f(x^{(t)})-v^{(t)}\|+\|x^{(t)}-H_t^{-1} \nabla f(x^{(t)})-x^*\| \\ & \leq \theta P^{(t)} + \alpha \|x^{(t)}-x^*\| \\ & = \theta \frac{\mu}{2 \sigma} K \alpha  \left(\frac{1+\alpha}{2} \right)^t + \alpha \frac{\mu}{2 \sigma} \alpha \left(\frac{1+\alpha}{2} \right)^t\\ &=(\theta K+\alpha) \frac{\mu}{2 \sigma} \alpha \left(\frac{1+\alpha}{2} \right)^t  \\ &= \frac{\mu}{2 \sigma} \alpha \left(\frac{1+\alpha}{2} \right)^{t+1}.
\end{align*}

which is what we need.

For the second inequality it suffices to prove that
\begin{equation*}
    \| \phi(H_{t+1}^i)-\phi(\nabla^2 f_i(x^{(t+1)})) \| \leq \frac{G^{(t+1)}}{2 \sqrt{2} \kappa}
\end{equation*}
and for that it suffices
\begin{equation*}
 \|\phi(\nabla^2 f_i(x^{(t+1)}))-\phi(H_t^i)\| \leq \frac{10 \sqrt{d}}{1+\alpha} G^{(t+1)}.   
\end{equation*}
We indeed have
\begin{align*}
    \|\phi(\nabla^2 f_i(x^{(t+1)}))-\phi(H_t^i)\| & \leq \|\phi(\nabla^2 f_i(x^{(t+1)}))-\phi(\nabla^2 f_i(x^{(t)}))+\phi(\nabla^2 f_i(x^{(t)}))-\phi(H_t^i) \| \\ & \leq \|\phi(\nabla^2 f_i(x^{(t+1)}))-\phi(\nabla^2 f_i(x^{(t)}))\|+\|\phi(\nabla^2 f_i(x^{(t)}))-\phi(H_t^i) \| \\ & \leq \sqrt{d} (\|\nabla^2 f_i(x^{(t+1)})-\nabla^2 f_i(x^{(t)})\|+\|\nabla^2 f_i(x^{(t)})-H_t^i\|)  \\ & \leq \sqrt{d} \left(\sigma \|x^{(t+1)}-x^{(t)}\|+\frac{ G^{(t)}}{\kappa} \right) \\ & \leq
    \sqrt{d}\left(2 \sigma \frac{\mu}{2 \sigma} \alpha \left(\frac{1+\alpha}{2} \right)^t
 +\frac{1}{\kappa} \frac{\mu}{4} \alpha \left(\frac{1+\alpha}{2} \right)^t \right)  \\ & \leq \sqrt{d} \frac{5 \mu}{4} \alpha \left(\frac{1+\alpha}{2} \right)^t = \sqrt{d} \frac{5 \mu}{4 \frac{1+\alpha}{2}} \alpha \left(\frac{1+\alpha}{2} \right)^{t+1} = \frac{10 \sqrt{d}}{1+\alpha} G^{(t+1)}.
\end{align*}
For the third inequality, we have
\begin{equation*}
    \|\nabla^2 f(x^{(t+1)})-S_{t+1}\| \leq \frac{1}{n} \sum_{i=1}^{n} \|\nabla^2 f_i(x^{(t+1)})-H_{t+1}^i\| \leq \frac{G^{(t+1)}}{2 \kappa}.
\end{equation*}
Now it suffices to prove
\begin{equation*}
    \|S_{t+1}-H_{t+1}\| \leq \frac{G^{(t+1)}}{2 \kappa}
\end{equation*}
which holds if
\begin{equation*}
    \|\phi(S_{t+1})-\phi(H_{t+1})\| \leq \frac{ G^{(t+1)}}{2 \sqrt{2} \kappa}
\end{equation*}
and for that suffices
\begin{equation*}
    \|\phi(S_{t+1})-\phi(H_t)\| \leq \sqrt{d} \left( \frac{1}{2 \kappa}+\frac{10}{1+\alpha} \right) G^{(t+1)}.
\end{equation*}
We now have
\begin{align*}
    \|\phi(S_{t+1})-\phi(H_t)\| &\leq \sqrt{d} \|S_{t+1}-H_t\| \leq \sqrt{d} \|S_{t+1}-\nabla^2 f(x^{(t+1)})+\nabla^2 f(x^{(t+1)})-\nabla^2 f(x^{(t)})+\nabla^2 f(x^{(t)})-H_t\| \\ & \leq \sqrt{d} (\|S_{t+1}-\nabla^2 f(x^{(t+1)})\|+\|\nabla^2 f(x^{(t+1)})-\nabla^2 f(x^{(t)})\|+\|\nabla^2 f(x^{(t)})-H_t\|) \\ & \leq \sqrt{d} \left (\frac{ G^{(t+1)}}{2 \kappa}+\sigma \|x^{(t+1)}-x^{(t)}\|+\frac{G^{(t)}}{\kappa} \right) \leq \sqrt{d} \frac{G^{(t+1)}}{2 \kappa}+\frac{10 \sqrt{d}}{1+\alpha} G^{(t+1)} \\ & =\sqrt{d} \left( \frac{1}{2 \kappa}+\frac{10}{1+\alpha} \right) G^{(t+1)}
\end{align*}
which concludes the induction.

For the fourth inequality it suffices to prove
\begin{equation*}
    \|H_{t+1}^{-1} \nabla f_i(x^{(t+1)})-v^{(t)}_i\| \leq 11 \kappa P^{(t+1)}.
\end{equation*}

To that end, we use $\gamma$-smoothness of $f_i$, the bound $\| x^{(t)}-x_i^* \| \leq \| x^{(t)}-x^* \|+\|x^*-x_i^*\| \leq \|x^{(0)}-x^*\|+\|x^{(0)}-x^*\|+\|x^{(0)}-x_i^*\| \leq \frac{3 \mu}{2 \sigma} \alpha $, the fact that $\|H_t^{-1}\|,\|H_{t+1}^{-1}\| \leq \frac{2}{\mu}$ and the induction hypothesis. Also, we use that $alpha<1$, $\kappa \geq 1$ and $K \geq 2$. 

Indeed, we have
\begin{align*}
&\|H_{t+1}^{-1} \nabla f_i(x^{(t+1)})-v^{(t)}_i\|  =   \|H_{t+1}^{-1} \nabla f_i(x^{(t+1)})-H_{t+1}^{-1} \nabla f_i(x^{(t)})+ H_{t+1}^{-1} \nabla f_i(x^{(t)})- H_t^{-1} \nabla f_i(x^{(t)}) + H_t^{-1} \nabla f_i(x^{(t)})- v^{(t)}_i \| \\ &\leq
\|H_{t+1}^{-1} \nabla f_i(x^{(t+1)})-H_{t+1}^{-1} \nabla f_i(x^{(t)})\|+\|H_{t+1}^{-1}-H_t^{-1}\| \| \nabla f_i(x^{(t)})\| + \| H_t^{-1} \nabla f_i(x^{(t)})-v^{(t)}_i \| \\ & \leq \frac{2}{\mu} \gamma \|x^{(t+1)}-x^{(t)}\|+ \|H_{t+1}^{-1}\| \| H_t^{-1}\| \| H_{t+1}-H_t\| \gamma_i \| x^{(t)}-x_i^*\|+ \theta P^{(t)} \\ & \leq
2 \frac{\gamma}{\mu} \|x^{(t+1)}-x^{(t)}\| +\frac{4}{\mu^2} (\| H_{t+1}-\nabla^2 f(x^{(t+1)}) \|+\|\nabla^2 f(x^{(t+1)}) -\nabla^2 f(x^{(t)})\|+\|\nabla^2 f(x^{(t)})-H_t\|) \gamma \frac{3 \mu}{2 \sigma} \alpha +\theta P^{(t)} \\ & \leq 2 \frac{\gamma}{\mu} \|x^{(t+1)}-x^{(t)}\|+\frac{4}{\mu^2} \left(\frac{G^{(t+1)}}{\kappa}+\frac{G^{(t)}}{\kappa} + \sigma \|x^{(t+1)}-x^{(t)}\| \right) \gamma \frac{3\mu}{2 \sigma} \alpha+ \theta P^{(t)} \\ & \leq
4 \kappa \frac{\mu}{2 \sigma} \alpha \left(\frac{1+\alpha}{2} \right)^t+ \frac{4}{\mu^2} \left(2 \frac{\mu}{4 \kappa} \alpha \left(\frac{1+\alpha}{2} \right)^t + 2 \sigma \frac{\mu}{2 \sigma} \alpha \left(\frac{1+\alpha}{2} \right)^t \right) \gamma \frac{3\mu}{2 \sigma} \alpha + \theta \frac{\mu}{2 \sigma} K \alpha \left( \frac{1+\alpha}{2}\right)^t  \\ & =
4 \kappa \frac{\mu}{2 \sigma} \alpha \left(\frac{1+\alpha}{2} \right)^t + \frac{12 \gamma}{\mu^2} \left( \frac{\mu}{2 \kappa} \alpha + \mu \alpha \right) \frac{\mu}{2 \sigma} \alpha \left(\frac{1+\alpha}{2} \right)^t + \theta \frac{\mu}{2 \sigma} K \alpha \left( \frac{1+\alpha}{2}\right)^t \\ & = 4 \kappa \frac{\mu}{2 \sigma} \alpha \left(\frac{1+\alpha}{2} \right)^t + 12 \kappa  \left( \frac{\alpha}{2 \kappa} + \alpha \right) \frac{\mu}{2 \sigma} \alpha \left(\frac{1+\alpha}{2} \right)^t + \theta \frac{\mu}{2 \sigma} K \alpha \left( \frac{1+\alpha}{2}\right)^t
\\ & \leq (4 \kappa+ 6+12\kappa+ \theta K) \frac{\mu}{2 \sigma} \alpha \left(\frac{1+\alpha}{2} \right)^t \leq (22 \kappa+\theta K) \frac{\mu}{2 \sigma} \alpha \left(\frac{1+\alpha}{2} \right)^t \leq 11\kappa(2/K+\theta)K \frac{\mu}{2 \sigma} \alpha \left(\frac{1+\alpha}{2} \right)^t\\&=11 \kappa K\alpha \frac{\mu}{2 \sigma} \left(\frac{1+\alpha}{2} \right)^{t+1} \leq 11 \kappa P^{(t+1)}.
\end{align*}

For the third inequality, we have
\begin{equation*}
    \|H_{t+1}^{-1} \nabla f(x^{(t+1)})-p^{(t+1)}\| \leq \frac{1}{n} \sum_{i=1}^n \| H_{t+1}^{-1} \nabla f_i(x^{(t+1)})- v^{(t+1)}_i\| \leq \frac{\theta P^{(t+1)}}{2}.
\end{equation*}
We want to prove also that
\begin{equation*}
    \|p^{(t+1)}-v^{(t+1)}\| \leq \frac{\theta P^{(t+1)}}{2}.
\end{equation*}
For that it suffices to show that
\begin{align*}
    \|p^{(t+1)}-v^{(t)}\| \leq \left(\frac{\theta}{2}+11 \kappa \right) P^{(t+1)}.
\end{align*}
We have
\begin{align*}
    \|p^{(t+1)}-v^{(t)}\| &\leq \|p^{(t+1)}-H_{t+1}^{-1} \nabla f(x^{(t+1)})+H_{t+1}^{-1} \nabla f(x^{(t+1)})-H_{t+1}^{-1} \nabla f(x^{(t)})\\ &+H_{t+1}^{-1} \nabla f(x^{(t)})-H_t^{-1} \nabla f(x^{(t)})+H_t^{-1} \nabla f(x^{(t)})-v^{(t)}\| \\ & \leq
    \|p^{(t+1)}-H_{t+1}^{-1} \nabla f(x^{(t+1)})\|+\|H_{t+1}^{-1} \nabla f(x^{(t+1)})-H_{t+1}^{-1} \nabla f(x^{(t)})\|\\&+\|H_{t+1}^{-1} \nabla f(x^{(t)})-H_t^{-1} \nabla f(x^{(t)})\|+\|H_t^{-1} \nabla f(x^{(t)})-v^{(t)}\| \\ & \leq \frac{\theta P^{(t+1)}}{2} + \frac{2}{\mu} \gamma \|x^{(t+1)}-x^{(t)}\|+ \|H_{t+1}^{-1}\| \| H_t^{-1}\| \| H_{t+1}-H_t\| \| \nabla f(x^{(t)})\|+ \theta P^{(t)} \\ & \leq \left(\frac{\theta}{2}+11 \kappa \right) P^{(t+1)} 
\end{align*}
which completes the induction by the same argument as in the previous derivation.
\end{proof}

\maintheorem*

\begin{proof}
The claim about the convergence of the iterates follows easily by applying Lemma \ref{le:desc_direction_newton} with $\alpha=\frac{1}{2}$.
\newline
This means that we achieve $\|x^{(t)}-x^*\| \leq \epsilon$ in at most
\begin{equation*}
    t=\frac{1}{1-\frac{3}{4}} \log \frac{ \frac{\mu}{4 \sigma} }{\epsilon}=4  \log \frac{\mu}{4 \sigma \epsilon}
\end{equation*}
many iterates.
We have $f(x^{(t)})-f^* \leq \epsilon$, if $\|x^{(t)}-x^*\| \leq \sqrt{\frac{2 \epsilon}{\gamma}}$, thus in at most
\begin{equation*}
    t=4 \log \frac{\gamma \mu^2}{32 \sigma^2 \epsilon}
\end{equation*}
many iterates.
\newline
For the communication cost, we have that in order to pursue Hessian quantization at $t=0$, we need
\begin{equation*}
    \bigO\left(n \frac{d(d+1)}{2} \log \frac{2 \sqrt{d} \gamma}{G^{(0)}/2\sqrt{2} \kappa}  \right) = \bigO\left(n \frac{d(d+1)}{2} \log \frac{2 \sqrt{d} \gamma}{\frac{\mu^2}{16 \sqrt{2} \gamma}}  \right) = \bigO\left(n \frac{d(d+1)}{2} \log \frac{\sqrt{d} \gamma^2}{\mu^2}  \right) \leq \bigO\left(n d^2 \log \left( \sqrt{d} \kappa  \right) \right)
\end{equation*}
many bits for encoding the local Hessian matrices
\begin{equation*}
    \bigO \left(n \frac{d(d+1)}{2} \log \frac{\sqrt{d} \left( \frac{1}{2 \kappa} G^{(0)} + 2 \gamma \right)}{ G^{(0)}/2\sqrt{2} \kappa} \right)= \bigO \left(n \frac{d(d+1)}{2} \log  \frac{2 \sqrt{d} \gamma} { G^{(0)}/2\sqrt{2} \kappa} \right) \leq \bigO\left(n d^2 \log \left( \sqrt{d} \kappa  \right) \right)
\end{equation*}
for decoding their sum back to all machines (this is because $\frac{1}{2\kappa} G^{(0)} \leq 2 \gamma$). Thus the total communication cost for Hessian quantization at $t=0$ is
\begin{equation*}
    \bigO\left(n d^2 \log \left( \sqrt{d} \kappa  \right) \right).
\end{equation*}
For $t \geq 1$, we have that the cost for quantizing the local Hessians is
\begin{equation*}
    \bigO \left(n \frac{d(d+1)}{2} \log\frac{10 \sqrt{d} G^{(t+1)}/(1+\alpha)}{G^{(t+1)}/2 \sqrt{2} \kappa} \right)= \bigO \left(n \frac{d(d+1)}{2} \log\frac{10 \sqrt{d} /(1+\alpha)}{1/2 \sqrt{2} \kappa} \right)=\bigO \left(n d^2 \log \left( \sqrt{d} \kappa \right) \right)
\end{equation*}
and for communicating the sum back to all machines is
\begin{equation*}
    \bigO \left(n \frac{d(d+1)}{2} \log\frac{\sqrt{d}(1/2 \kappa+10/(1+\alpha)) G^{(t+1)}}{G^{(t+1)}/2 \sqrt{2} \kappa} \right)=\bigO \left(n \frac{d(d+1)}{2} \log\frac{10 \sqrt{d}/(1+\alpha) }{1/2 \sqrt{2} \kappa} \right)=\bigO \left(n d^2 \log \left( \sqrt{d} \kappa \right) \right),
\end{equation*}
again because $1/2 \kappa \leq 10/ (1+\alpha)$.
\newline
Thus the total cost of Hessian quantization along the whole optimization process until reaching accuracy $\epsilon$ is
\begin{equation*}
    b_m=\bigO \left(n d^2 \log \left( \sqrt{d} \kappa \right) \log \frac{\gamma \mu^2}{32 \sigma^2 \epsilon} \right)
\end{equation*}
many bits in total.
\newline
On the other hand, the cost of quantizing the local descent directions at $t\geq0$ is
\begin{equation*}
    \bigO \left(nd \log \frac{11 \kappa P^{(t)}}{\frac{\theta P^{(t)}}{2}} \right)= \bigO \left(nd\log \kappa \right)
\end{equation*}
because now $\theta$ is just $\frac{1}{16}$. The cost of sending the average of the quantized local directions back to any machine is
\begin{equation*}
    \bigO \left(nd \log \frac{(\theta/2+11
    \kappa)P^{(t)}}{\theta P^{(t)}/2} \right)=\bigO \left(nd \log \frac{11 \kappa P^{(t)}}{\frac{\theta P^{(t)}}{2}} \right)=\bigO \left(nd\log \kappa \right),
\end{equation*}
because $\frac{\theta}{2} \leq 11 \kappa$. Thus, the total communication cost for quantizing the descent directions until reaching accuracy $\epsilon$ is
\begin{equation*}
    b=\bigO \left(nd\log \kappa \log \frac{\gamma \mu^2}{32 \sigma^2 \epsilon} \right)
\end{equation*}
many bits.
\newline
The total communication cost of Quantized Newton's method overall is
\begin{equation*}
    \bigO \left(n d^2 \log \left( \sqrt{d} \kappa \right) \log \frac{\gamma \mu^2}{32 \sigma^2 \epsilon} \right)+\bigO \left(nd\log \kappa \log \frac{\gamma \mu^2}{32 \sigma^2 \epsilon} \right) = \bigO\left(nd^2 \log \left( \sqrt{d} \kappa \right) \log \frac{\gamma \mu^2}{\sigma^2 \epsilon} \right).
\end{equation*}
\end{proof}

\section{Estimation of the Minimum}
\label{app:function_value}

\functionvalue*

\begin{proof}
We have that
\begin{align*}
    &\mid f_i(x^{(t)})-f_{i_0}(x^{(t)}) \mid \leq \mid f_i(x^{(t)}) \mid + \mid  f_{i_0}(x^{(t)}) \mid \leq \frac{\gamma}{2} \|x^{(t)}-x_i^*\|^2 + \mid f_i^* \mid+ \frac{\gamma}{2} \|x^{(t)}-x_{i_0}^*\|^2+\mid f_{i_0}^* \mid
\end{align*}
In order $x^{(t)}$ to satisfy $f(x^{(t)})-f^* \leq \frac{\epsilon}{2}$, we compute $x^{(t)}$ by our main algorithms, such that $\|x^{(t)}-x^*\| \leq \sqrt{\frac{\epsilon}{\gamma}}$. This gives the respective communication complexities from the previous sections.
\newline
Given that, we can write
\begin{align*}
    &\|x^{(t)}-x_i^*\|^2 =\|x^*-x_i^*\|^2 + \|x^{(t)}-x^*\|^2+2 \langle x^*-x_i^*, x^{(t)}-x^* \rangle \leq \\ &  \|x^*-x_i^*\|^2 + \|x^{(t)}-x^*\|^2+2 \| x^*-x_i^* \| \| x^{(t)}-x^* \|  \leq \|x^*-x_i^*\|^2+\frac{\epsilon}{\gamma}+\sqrt{\frac{\epsilon}{\gamma}} \|x^*-x_i^*\| \leq \\ & C^2 +\frac{\epsilon}{\gamma}+\sqrt{\frac{\epsilon}{\gamma}} C \leq 2 C^2
\end{align*}
for sufficiently small $\epsilon$. Similarly we have 
\begin{equation*}
    \|x^{(T)}-x_{i_0}^*\|^2 \leq 2 C^2
\end{equation*}
for small $\epsilon$.
\newline
Thus 
\begin{equation*}
    \mid f_i(x^{(t)})-f_{i_0}(x^{(t)}) \mid \leq 2 (\gamma C^2+c)
\end{equation*}
and by the definition of the quantization, we have
\begin{equation*}
    \mid q_i^{(t)}-f_i(x^{(t)}) \mid \leq \frac{\epsilon}{2}.
\end{equation*}
which implies
\begin{equation*}
    \mid \bar f-f(x^{(t)}) \mid \leq \frac{1}{n} \sum_{i=1}^n \mid q_i^{(t)}-f_i(x^{(t)}) \mid \leq \frac{\epsilon}{2}.
\end{equation*}
Overall, we get
\begin{equation*}
    \bar f-f^* \leq \mid \bar f-f(x^{(t)}) \mid + f(x^{(t)})-f^* \leq \frac{\epsilon}{2}+\frac{\epsilon}{2}= \epsilon.
\end{equation*}
The communication cost for quantizing $f_i(x^{(t)})$ is
\begin{equation*}
    \bigO \left(n \log \frac{\gamma C^2+c}{\epsilon} \right)
\end{equation*}
since we quantize real numbers, which are $1$-dimensional, and we need to communicate $n$-times.
\end{proof}

%% file: main_paper.bbl
\begin{thebibliography}{34}
\providecommand{\natexlab}[1]{#1}
\providecommand{\url}[1]{\texttt{#1}}
\expandafter\ifx\csname urlstyle\endcsname\relax
  \providecommand{\doi}[1]{doi: #1}\else
  \providecommand{\doi}{doi: \begingroup \urlstyle{rm}\Url}\fi

\bibitem[Alistarh \& Korhonen(2020)Alistarh and Korhonen]{alistarh2020improved}
Alistarh, D. and Korhonen, J.~H.
\newblock Improved communication lower bounds for distributed optimisation.
\newblock \emph{arXiv preprint arXiv:2010.08222}, 2020.

\bibitem[Alistarh et~al.(2016)Alistarh, Grubic, Li, Tomioka, and
  Vojnovic]{alistarh2016qsgd}
Alistarh, D., Grubic, D., Li, J., Tomioka, R., and Vojnovic, M.
\newblock Qsgd: Communication-efficient sgd via gradient quantization and
  encoding.
\newblock \emph{arXiv preprint arXiv:1610.02132}, 2016.

\bibitem[Alistarh et~al.(2018)Alistarh, Hoefler, Johansson, Khirirat,
  Konstantinov, and Renggli]{alistarh2018convergence}
Alistarh, D., Hoefler, T., Johansson, M., Khirirat, S., Konstantinov, N., and
  Renggli, C.
\newblock The convergence of sparsified gradient methods.
\newblock \emph{arXiv preprint arXiv:1809.10505}, 2018.

\bibitem[Arjevani \& Shamir(2015)Arjevani and Shamir]{NIPS2015_5731}
Arjevani, Y. and Shamir, O.
\newblock Communication complexity of distributed convex learning and
  optimization.
\newblock In \emph{Advances in Neural Information Processing Systems 28 (NIPS
  2015)}, pp.\  1756--1764, 2015.

\bibitem[Ben-Nun \& Hoefler(2019)Ben-Nun and Hoefler]{ben2019demystifying}
Ben-Nun, T. and Hoefler, T.
\newblock Demystifying parallel and distributed deep learning: An in-depth
  concurrency analysis.
\newblock \emph{ACM Computing Surveys (CSUR)}, 52\penalty0 (4):\penalty0 1--43,
  2019.

\bibitem[Chang \& Lin(2011)Chang and Lin]{LibSVM}
Chang, C.-C. and Lin, C.-J.
\newblock {LIBSVM}: A library for support vector machines.
\newblock \emph{ACM Transactions on Intelligent Systems and Technology},
  2:\penalty0 27:1--27:27, 2011.
\newblock Software available at \url{http://www.csie.ntu.edu.tw/~cjlin/libsvm}.

\bibitem[Chen(2019)]{chen2019gradient}
Chen, Y.
\newblock Gradient methods for unconstrained problems.
\newblock
  \url{http://www.princeton.edu/~yc5/ele522_optimization/lectures/grad_descent_unconstrained.pdf},
  2019.
\newblock Princeton University, Fall 2019.

\bibitem[Davies et~al.(2021)Davies, Gurunanthan, Moshrefi, Ashkboos, and
  Alistarh]{davies2021new}
Davies, P., Gurunanthan, V., Moshrefi, N., Ashkboos, S., and Alistarh, D.
\newblock New bounds for distributed mean estimation and variance reduction.
\newblock In \emph{International Conference on Learning Representations}, 2021.
\newblock URL \url{https://openreview.net/forum?id=t86MwoUCCNe}.

\bibitem[Ghadikolaei \& Magn{\'u}sson(2020)Ghadikolaei and
  Magn{\'u}sson]{ghadikolaei2020communication}
Ghadikolaei, H.~S. and Magn{\'u}sson, S.
\newblock Communication-efficient variance-reduced stochastic gradient descent.
\newblock \emph{arXiv preprint arXiv:2003.04686}, 2020.

\bibitem[Hendrikx et~al.(2020)Hendrikx, Xiao, Bubeck, Bach, and
  Massoulie]{hendrikx2020statistically}
Hendrikx, H., Xiao, L., Bubeck, S., Bach, F., and Massoulie, L.
\newblock Statistically preconditioned accelerated gradient method for
  distributed optimization.
\newblock In \emph{International Conference on Machine Learning}, pp.\
  4203--4227. PMLR, 2020.

\bibitem[Islamov et~al.(2021)Islamov, Qian, and
  Richt{\'a}rik]{islamov2021distributed}
Islamov, R., Qian, X., and Richt{\'a}rik, P.
\newblock Distributed second order methods with fast rates and compressed
  communication.
\newblock \emph{arXiv preprint arXiv:2102.07158. Accepted to ICML 2021.}, 2021.

\bibitem[Jaggi et~al.(2014)Jaggi, Smith, Tak{\'a}{\v{c}}, Terhorst, Krishnan,
  Hofmann, and Jordan]{jaggi2014communication}
Jaggi, M., Smith, V., Tak{\'a}{\v{c}}, M., Terhorst, J., Krishnan, S., Hofmann,
  T., and Jordan, M.~I.
\newblock Communication-efficient distributed dual coordinate ascent.
\newblock \emph{arXiv preprint arXiv:1409.1458}, 2014.

\bibitem[Jordan et~al.(2018)Jordan, Lee, and Yang]{jordan2018communication}
Jordan, M.~I., Lee, J.~D., and Yang, Y.
\newblock Communication-efficient distributed statistical inference.
\newblock \emph{Journal of the American Statistical Association}, 2018.

\bibitem[Khirirat et~al.(2018)Khirirat, Feyzmahdavian, and
  Johansson]{khirirat2018distributed}
Khirirat, S., Feyzmahdavian, H.~R., and Johansson, M.
\newblock Distributed learning with compressed gradients.
\newblock \emph{arXiv preprint arXiv:1806.06573}, 2018.

\bibitem[Li et~al.(2014)Li, Andersen, Park, Smola, Ahmed, Josifovski, Long,
  Shekita, and Su]{PS}
Li, M., Andersen, D.~G., Park, J.~W., Smola, A.~J., Ahmed, A., Josifovski, V.,
  Long, J., Shekita, E.~J., and Su, B.-Y.
\newblock Scaling distributed machine learning with the parameter server.
\newblock In \emph{Proc.\ 11th {USENIX} Symposium on Operating Systems Design
  and Implementation ({OSDI} 2014)}, pp.\  583--598, 2014.

\bibitem[Magn{\'u}sson et~al.(2020)Magn{\'u}sson, Shokri-Ghadikolaei, and
  Li]{magnusson2020maintaining}
Magn{\'u}sson, S., Shokri-Ghadikolaei, H., and Li, N.
\newblock On maintaining linear convergence of distributed learning and
  optimization under limited communication.
\newblock \emph{IEEE Transactions on Signal Processing}, 68:\penalty0
  6101--6116, 2020.

\bibitem[Mendler-D{\"u}nner \& Lucchi(2020)Mendler-D{\"u}nner and
  Lucchi]{mendler2020randomized}
Mendler-D{\"u}nner, C. and Lucchi, A.
\newblock Randomized block-diagonal preconditioning for parallel learning.
\newblock In \emph{International Conference on Machine Learning}, pp.\
  6841--6851. PMLR, 2020.

\bibitem[Nesterov \& Polyak(2006)Nesterov and Polyak]{RePEc:cor:louvrp:1927}
Nesterov, Y. and Polyak, B.
\newblock Cubic regularization of newton method and its global performance.
\newblock LIDAM Reprints CORE 1927, Université catholique de Louvain, Center
  for Operations Research and Econometrics (CORE), 2006.
\newblock URL \url{https://EconPapers.repec.org/RePEc:cor:louvrp:1927}.

\bibitem[Nguyen et~al.(2018)Nguyen, Nguyen, Dijk, Richt{\'a}rik, Scheinberg,
  and Tak{\'a}c]{nguyen2018sgd}
Nguyen, L., Nguyen, P.~H., Dijk, M., Richt{\'a}rik, P., Scheinberg, K., and
  Tak{\'a}c, M.
\newblock Sgd and hogwild! convergence without the bounded gradients
  assumption.
\newblock In \emph{International Conference on Machine Learning}, pp.\
  3750--3758. PMLR, 2018.

\bibitem[Niu et~al.(2011)Niu, Recht, R{\'e}, and Wright]{niu2011hogwild}
Niu, F., Recht, B., R{\'e}, C., and Wright, S.~J.
\newblock Hogwild!: A lock-free approach to parallelizing stochastic gradient
  descent.
\newblock \emph{arXiv preprint arXiv:1106.5730}, 2011.

\bibitem[Ramezani{-}Kebrya et~al.(2019)Ramezani{-}Kebrya, Faghri, and
  Roy]{NUQSGD}
Ramezani{-}Kebrya, A., Faghri, F., and Roy, D.~M.
\newblock {NUQSGD:} improved communication efficiency for data-parallel {SGD}
  via nonuniform quantization.
\newblock \emph{CoRR}, abs/1908.06077, 2019.
\newblock URL \url{http://arxiv.org/abs/1908.06077}.

\bibitem[Reddi et~al.(2016)Reddi, Kone{\v{c}}n{\`y}, Richt{\'a}rik,
  P{\'o}cz{\'o}s, and Smola]{reddi2016aide}
Reddi, S.~J., Kone{\v{c}}n{\`y}, J., Richt{\'a}rik, P., P{\'o}cz{\'o}s, B., and
  Smola, A.
\newblock Aide: Fast and communication efficient distributed optimization.
\newblock \emph{arXiv preprint arXiv:1608.06879}, 2016.

\bibitem[Safaryan et~al.(2021)Safaryan, Islamov, Qian, and
  Richtárik]{safaryan2021fednl}
Safaryan, M., Islamov, R., Qian, X., and Richtárik, P.
\newblock Fednl: Making newton-type methods applicable to federated learning,
  2021.

\bibitem[Scaman et~al.(2017)Scaman, Bach, Bubeck, Lee, and
  Massouli{\'e}]{scaman2017optimal}
Scaman, K., Bach, F., Bubeck, S., Lee, Y.~T., and Massouli{\'e}, L.
\newblock Optimal algorithms for smooth and strongly convex distributed
  optimization in networks.
\newblock In \emph{international conference on machine learning}, pp.\
  3027--3036. PMLR, 2017.

\bibitem[Shamir(2014)]{NIPS2014_5386}
Shamir, O.
\newblock Fundamental limits of online and distributed algorithms for
  statistical learning and estimation.
\newblock In \emph{Advances in Neural Information Processing Systems 27 (NIPS
  2014)}, pp.\  163--171, 2014.

\bibitem[Shamir et~al.(2014)Shamir, Srebro, and Zhang]{shamir2014communication}
Shamir, O., Srebro, N., and Zhang, T.
\newblock Communication-efficient distributed optimization using an approximate
  newton-type method.
\newblock In \emph{International conference on machine learning}, pp.\
  1000--1008. PMLR, 2014.

\bibitem[Suresh et~al.(2017)Suresh, Yu, Kumar, and McMahan]{pmlr-v70-suresh17a}
Suresh, A.~T., Yu, F.~X., Kumar, S., and McMahan, H.~B.
\newblock Distributed mean estimation with limited communication.
\newblock In Precup, D. and Teh, Y.~W. (eds.), \emph{Proceedings of the 34th
  International Conference on Machine Learning}, volume~70 of \emph{Proceedings
  of Machine Learning Research}, pp.\  3329--3337, 2017.

\bibitem[{Tsitsiklis} \& {Luo}(1986){Tsitsiklis} and {Luo}]{4048825}
{Tsitsiklis}, J.~N. and {Luo}, Z.
\newblock Communication complexity of convex optimization.
\newblock In \emph{1986 25th IEEE Conference on Decision and Control}, pp.\
  608--611, 1986.
\newblock \doi{10.1109/CDC.1986.267379}.

\bibitem[Vempala et~al.(2020)Vempala, Wang, and
  Woodruff]{vempala2020communication}
Vempala, S.~S., Wang, R., and Woodruff, D.~P.
\newblock The communication complexity of optimization.
\newblock In \emph{Proceedings of the Fourteenth Annual ACM-SIAM Symposium on
  Discrete Algorithms}, pp.\  1733--1752. SIAM, 2020.

\bibitem[Wang et~al.(2018)Wang, Roosta, Xu, and Mahoney]{wang2018giant}
Wang, S., Roosta, F., Xu, P., and Mahoney, M.~W.
\newblock Giant: Globally improved approximate newton method for distributed
  optimization.
\newblock \emph{Advances in Neural Information Processing Systems},
  31:\penalty0 2332--2342, 2018.

\bibitem[Ye \& Abbe(2018)Ye and Abbe]{ye2018communication}
Ye, M. and Abbe, E.
\newblock Communication-computation efficient gradient coding.
\newblock In \emph{International Conference on Machine Learning}, pp.\
  5610--5619. PMLR, 2018.

\bibitem[Zhang et~al.(2020)Zhang, You, and Ba{\c{s}}ar]{zhang2020distributed}
Zhang, J., You, K., and Ba{\c{s}}ar, T.
\newblock Distributed adaptive newton methods with globally superlinear
  convergence.
\newblock \emph{arXiv preprint arXiv:2002.07378}, 2020.

\bibitem[Zhang \& Lin(2015)Zhang and Lin]{zhang2015disco}
Zhang, Y. and Lin, X.
\newblock Disco: Distributed optimization for self-concordant empirical loss.
\newblock In \emph{International conference on machine learning}, pp.\
  362--370. PMLR, 2015.

\bibitem[Zhang et~al.(2013)Zhang, Duchi, Jordan, and Wainwright]{NIPS2013_4902}
Zhang, Y., Duchi, J., Jordan, M.~I., and Wainwright, M.~J.
\newblock Information-theoretic lower bounds for distributed statistical
  estimation with communication constraints.
\newblock In \emph{Advances in Neural Information Processing Systems 26 (NIPS
  2013)}, pp.\  2328--2336, 2013.

\end{thebibliography}
